\documentclass[12pt,reqno]{amsart}
\usepackage{amscd,amsmath,amsthm,amssymb}
\usepackage{pgf,tikz}
\usetikzlibrary{arrows}
\usepackage{ifpdf}

\numberwithin{equation}{section}

\def\frk{\frak}               

\def\Phi{{\frk n}}
\def\Phi{{\frk N}}
\def\opn#1#2{\def#1{\operatorname{#2}}} 
\opn\chara{char} \opn\length{\ell} \opn\pd{pd} \opn\rk{rk}
\opn\projdim{proj\,dim} \opn\injdim{inj\,dim} \opn\rank{rank}
\opn\depth{depth} \opn\grade{grade} \opn\height{height}
\opn\embdim{emb\,dim} \opn\codim{codim}

\opn\Tr{Tr} \opn\bigrank{big\,rank}
\opn\superheight{superheight}\opn\lcm{lcm}
\opn\trdeg{tr\,deg}
\opn\reg{reg} \opn\lreg{lreg} \opn\ini{in} \opn\lpd{lpd}
\opn\size{size}\opn\bigsize{bigsize}
\opn\cosize{cosize}\opn\bigcosize{bigcosize}
\opn\sdepth{sdepth}\opn\sreg{sreg}
\opn\link{link}\opn\fdepth{fdepth}
\opn\index{index}
\opn\index{index}
\opn\indeg{indeg}
\opn\N{N}
\opn\SSC{SSC}
\opn\SC{SC}
\opn\lk{lk}
\opn\div{div} \opn\Div{Div} \opn\cl{cl} \opn\Cl{Cl}
\opn\Spec{Spec} \opn\Supp{Supp} \opn\supp{supp} \opn\Sing{Sing}
\opn\Ass{Ass} \opn\Min{Min}\opn\Mon{Mon} \opn\dstab{dstab} \opn\astab{astab}
\opn\Syz{Syz}
\opn\reg{reg}
\opn\Ann{Ann} \opn\Rad{Rad} \opn\Soc{Soc}
\opn\Im{Im} \opn\Ker{Ker} \opn\Coker{Coker} \opn\Am{Am}
\opn\Hom{Hom} \opn\Tor{Tor} \opn\Ext{Ext} \opn\End{End}\opn\Der{Der}
\opn\Aut{Aut} \opn\id{id}

\opn\nat{nat}
\opn\pff{pf}
\opn\Pf{Pf} \opn\GL{GL} \opn\SL{SL} \opn\mod{mod} \opn\ord{ord}
\opn\Gin{Gin} \opn\Hilb{Hilb}\opn\sort{sort}
\opn\initial{init}
\opn\ende{end}
\opn\height{height}
\opn\type{type}
\opn\aff{aff} \opn\con{conv} \opn\relint{relint} \opn\st{st}
\opn\lk{lk} \opn\cn{cn} \opn\core{core} \opn\vol{vol}
\opn\link{link} \opn\Link{Link}\opn\lex{lex}
\opn\gr{gr}

\def\pot#1#2{#1[\kern-0.28ex[#2]\kern-0.28ex]}

\opn\dirlim{\underrightarrow{\lim}}
\opn\inivlim{\underleftarrow{\lim}}
\def\Implies{\ifmmode\Longrightarrow \else
        \unskip${}\Longrightarrow{}$\ignorespaces\fi}
\def\implies{\ifmmode\Rightarrow \else
        \unskip${}\Rightarrow{}$\ignorespaces\fi}
\def\iff{\ifmmode\Longleftrightarrow \else
        \unskip${}\Longleftrightarrow{}$\ignorespaces\fi}

\let\:=\colon
\newtheorem{Theorem}{Theorem}[section]
 \newtheorem{Lemma}[Theorem]{Lemma}
 \newtheorem{Corollary}[Theorem]{Corollary}
 \newtheorem{Proposition}[Theorem]{Proposition}
 \newtheorem{Remark}[Theorem]{Remark}
 
 \newtheorem{Example}[Theorem]{Example}
 
 \newtheorem{Definition}[Theorem]{Definition}

\let\epsilon\varepsilon
\let\kappa=\varkappa
\def\qed{\ifhmode\textqed\fi
      \ifmmode\ifinner\quad\qedsymbol\else\dispqed\fi\fi}
\def\textqed{\unskip\nobreak\penalty50
       \hskip2em\hbox{}\nobreak\hfil\qedsymbol
       \parfillskip=0pt \finalhyphendemerits=0}
\def\dispqed{\rlap{\qquad\qedsymbol}}

\opn\dis{dis}
\def\pnt{{\raise0.5mm\hbox{\large\bf.}}}

\opn\Lex{Lex}

\begin{document}

\title{Monomial ideals with  regular quotients and some edge rings}

\author  {Dancheng Lu, Hao Zhou$^*$}

\thanks{*Corresponding author. }

\address{School  of Mathematical Sciences, Soochow University, 215006 Suzhou, P.R.China}
\email{ludancheng@suda.edu.cn, zhouhao9939@outlook.com,}

\keywords{Betti numbers, toric ring, regular quotients, bipartite graphs, initial ideals }

\subjclass[2010]{Primary 13D45, 13D05; Secondary 13D02.}

\begin{abstract} We introduce and study monomial ideals with regular quotients, which can be seen as an extension of monomial ideals with linear quotients. Based on these investigations, we are able to calculate the Betti numbers of toric ideals belonging to various classes of edge rings.
\end{abstract}
\maketitle
\section*{Introduction}
Let $\mathbb{K}$ be a field, $R=\mathbb{K}[x_1,\ldots,x_n]$ the polynomial ring in $n$ variables. The concept of (monomial) ideals with linear quotients was introduced by Herzog and Takayama \cite{HT02}.  This notion is useful in many ways, specifically, it is an effective method for proving that an ideal has a linear resolution and for constructing a minimal graded free resolution for some ideals, see e.g. \cite{HHZ04} and \cite{C}.  In this paper, we embark on a study of monomial ideals with regular quotients, which can be viewed as a natural extension of monomial ideals possessing linear quotients.  Let $I\subset R$ be a monomial ideal. We say $I$ has {\it regular quotients}  with respect to the order $f_1,\ldots, f_m$ of minimal generators of $I$ if the quotient
ideal $(f_1,\ldots, f_{i-1}): f_i$
is generated by a regular sequence on $R$ for all $i=2,\ldots, m.$ (notice
that this property depends on the order of the generators).

Let $I$ be a monomial ideal with regular quotients.  In the most general scenario, we provide upper bounds for the regularity and projective dimension of $I$. Furthermore, we explore the conditions under which these upper bounds are attained, see Corollary~\ref{2.5}. To study the Betti numbers of such ideals, we specifically examine the ideals where each quotient is generated by a regular sequence of the same degree.  Let $I=(f_1,\ldots,f_m)$ be a monomial ideal  such that the quotient $L_k:=(f_1,\ldots,f_{k-1}):f_k$ is generated by a regular sequence of length $r_k$ in which every element has degree $a_k$ for $k=2,\ldots,m$.   Denote $\deg(f_k)$ by $d_k$ for $k=1,\ldots,m$. By the construction of mapping cones, we obtain the following upper bound of $\beta_{i,j}(I)$: $$ \beta_{i,j}(I)\leq \beta_{i,j}(R/L_m[-d_m])+\cdots+\beta_{i,j}(R/L_2[-d_2])+\beta_{i,j}(R[-d_1]).$$
  Note that the right-hand side represents a number that is determined by the values of $d_i$, $r_i$ and $a_i$. If $a_2=\cdots=a_m=1$, i.e. $I$ has linear quotients, the upper bound given above is actually the true value of $\beta_{i,j}(I)$, see \cite[Lemma 2.1 and Proposition 2.2]{SJZ10} or \cite[Corollary 2.7]{SV}. In this paper, we show in Remark~\ref{r=1} that this is also the case if $r_2=\cdots=r_m=1$.  In  general, we give a sufficient condition in terms of a relation among $d_i$'s and $a_i$'s for which the upper upper is achieved, see Theorem~\ref{main1} and Corollary~\ref{main11}.

Let $G=(V,E)$ be a simple graph with vertex set $V=\{x_1,\ldots,x_n\}$ and edge set $E=\{e_1,\ldots,e_r\}$. The {\it edge ring} $\mathbb{K}[G]$ is the toric ring $\mathbb{K}[x_e\:\; e\in E(G)]\subset \mathbb{K}[x_1,\ldots,x_n]$. Here, $x_e=\prod_{x_i\in e}x_i$ for all $e\in E(G)$. Let $\mathbb{K}[E]$ be the polynomial ring $\mathbb{K}[e_1,\ldots,e_r]$. Then there is exactly one ring homomorphism
$$\phi:  \mathbb{K}[E]\rightarrow \mathbb{K}[V] \quad {\mbox{ such that }} ~e_i\mapsto x_{e_i} ~{i=1,\ldots,r}$$
The kernel of the homomorphism map $\phi$ is called the {\it toric ideal} or the {\it defining ideal} of $\mathbb{K}[G]$ or $G$, which is denoted by $I_G$. It follows  that $\mathbb{K}[G]\cong \mathbb{K}[E]/I_G$.  In general, calculating the Betti numbers or other invariants of a toric ideal for a specific type of graphs is challenging. For recent advancements in this area, please refer to \cite{BOV}, \cite{FKV}, \cite{GHK} and \cite{HBO}, etc.

 In this paper, we introduce two classes of simple graphs, for which the toric ideals have an initial ideal with regular quotients.  The first class of graphs is a class of  bipartite graphs, which we denote by $B_{\underline{\ell},h}$, where $h\geq 2$ and $\underline{\ell}=(\ell_1,\ldots,\ell_h)\in \mathbb{Z}_{\geq 1}^h.$ We  calculate the regularity and projective dimension of the edge ring of this class of graphs, and in the case when $\ell_2=\cdots=\ell_{h}$, we  also compute the Betti numbers of the edge ring of these graphs. As a byproduct, we demonstrate the existence of a bipartite graph $G$ such that  $\mathrm{reg}(I_G)=r$ and $\mathrm{pdim}(I_G)=p$, where $p,r$ are any given integers with $p\geq 0$ and $r\geq 2$. We refer to \cite{NH} for the background of this result. The second class of graphs is obtained by attaching two odd cycles to the graph $B_{\underline{\ell},h}$ when $\ell_1=\cdots=\ell_h$. For such graphs, we also compute the Betti numbers of their edge rings. In all of these computations, we rely on a sufficient condition that ensures the toric ideal and its initial ideal share the same Betti numbers. We present this condition  in Theorem~\ref{bridge}, which is inspired by the proof of \cite[Corollary 3.6]{FKV}.

The paper is organized as follows. Section 1 provides a brief overview of toric ideals of graphs and some homological invariants. In Section 2, we derive a formula for the graded Betti numbers of a monomial ideal, assuming that its successive colon ideals are generated by a regular sequence. Section 3 presents a sufficient condition for the graded Betti numbers of the toric ideal of a simple  graph to be equal to the graded Betti numbers of its initial ideal. In Sections 4 and 5, we introduce two classes of  graphs respectively and compute the graded Betti numbers and other invariants of their toric ideals, utilizing the results obtained in the preceding sections.
\section{Preliminaries}
In this section, we provide a brief review of the notation and fundamental facts that will be utilized later on.
\subsection{Toric ideals of graphs}
\hspace*{\fill} \\
Let $G$ be a simple graph, i.e., a finite graph without loops and multiple edges. Recall that a walk of $G$ of length $q$ is a subgraph $W$ of $G$ such that $E(W)=\{\{v_0,v_1\},\{v_1,v_2\},\ldots, \{v_{q-1},v_q\}\}$, where $v_0,v_1,\ldots,v_q$ are vertices of $G$. An walk $W$ of $G$ is even if $q$ is even, and it is closed if $v_0=v_q$. A cycle with edge set $\{\{v_0,v_1\},\{v_1,v_2\},\{v_{q-1},v_q=v_0\}\}$ is a special even closed walk where $v_1,\ldots,v_q$ are pairwise distinct.

The generators of the toric ideal of $I_G$ are binomials which is tightly related to even closed walks in $G$. Given an even closed walk $W$ of $G$ with
$$E(W)=\{\{v_0,v_1\},\{v_1,v_2\},\ldots,\{v_{2q-2},v_{2q-1}\},\{v_{2q-1},v_0\}\}$$
We associate $W$ with the binomial defined by
$$f_W:=\prod\limits_{j=1}^{q}e_{2j-1}-\prod\limits_{j=1}^{q}e_{2j},$$
where $e_j=\{v_{j-1},v_j\}$ for $1\leq j\leq 2q-1$ and $e_{2q}=\{v_{2q-1},v_0\}$. A binomial $f = u - v \in I_{G}$ is called a {\it primitive binomial} if there is no binomial $g = u' - v' \in I_{G}$ such that $u'|u$ and $v'|v$. An even closed walk $W$ of $G$ is a {\it primitive even closed walk} if its associated binomial $f_{W}$ is a primitive binomial in $I_G$.

In this paper, we use $\mathcal{PC}(G)$  for the set of primitive even closed walks of $G$. Then $\mathcal{PC}(G)$  is a finite set, and the set  $$\{f_W\:\; W\in \mathcal{PC}(G)\}$$ is the universal  Gr$\ddot{\mathrm{o}}$bner base of $I_G$ by e.g. \cite[Proposition~5.19]{EH}.
 In particular, it is a Gr$\ddot{\mathrm{o}}$bner base of $I_G$ with respect to any monomial order. The set of primitive even walks of a graph $G$ was described in \cite{OH} explicitly.
\begin{Proposition} {\em (\cite[Lemma~3.2]{OH})} \label{primitive}
Let $G$ be a simple connected graph. Then a primitive even closed walk $\Gamma$ of $G$ is one of the following:
\begin{enumerate}
  \item $\Gamma$ is an even cycle of $G$;
  \item $\Gamma$=$(C_1,C_2)$, where $C_1$ and $C_2$ are odd cycles of $G$ having exactly one common vertex;
  \item $\Gamma$=$(C_1,\Gamma_1,C_2,\Gamma_2)$,  where $C_1$ and $C_2$ are odd cycles of $G$ having no common vertex and where $\Gamma_1$ and $\Gamma_2$ are walks of $G$ both of which combine a vertex $v_1$ of $C_1$ and a vertex $v_2$ of $C_2$.
\end{enumerate}
\end{Proposition}
\subsection{Homological invariants}
\hspace*{\fill} \\
Let $R:=\mathbb{K}[x_1,\ldots,x_n]$ be the polynomial ring in variables $x_1,\ldots,x_n$, which is standard graded. For a finitely generated graded  $R-$module $M$, there exists the minimal graded free resolution of $M$ that has the form:
$$0\rightarrow\underset{j\in\mathbb{Z}}{\bigoplus}R[-j]^{\beta_{p,j}(M)}\rightarrow \cdots \rightarrow \underset{j\in\mathbb{Z}}{\bigoplus}R[-j]^{\beta_{1,j}(M)}\rightarrow\underset{j\in\mathbb{Z}}{\bigoplus}R[-j]^{\beta_{0,j}(M)}\rightarrow M \rightarrow 0.$$
Here, $R[-j]$ is the cyclic free $R$-module generated in degree $j$ and $\beta_{i,j}(M)={\rm{dim}}_{\mathbb{K}}\mathrm{Tor}_i^R(M,\mathbb{K})_j$ is called the $(i,j)$-th graded Betti number of $M$. A Betti number $\beta_{i,i+j}(M)\neq 0$ is called an {\it extremal Betti number} of $M$ if $\beta_{k,k+\ell}(M)=0$ for any pair $(k,\ell)\neq (i,j)$ with $k\geq i$ and $\ell\geq j$.
Many homological invariants of $M$ can be defined in terms of  its minimal graded free resolution. The {\em{Castelnuovo-Mumford regularity}} of $M$ is
$$\mbox{reg}\,(M):={\mbox{max}}\,\{j-i\:\; \beta_{i,\,j}(M)\neq 0\}.$$
The {\rm{projective dimension}} of $M$ is the length of the minimal graded free resolution:
$$\mbox{pdim}\,(M):=\mbox{max}\,\{i\:\;\ \beta_{i,\,j}(M)\neq 0\}.$$
The {\em{Hilbert series}} of  $M$ is defined to the  power series:
$$HS(M,t):=\underset {i\in \mathbb{Z}}{\sum}[\mbox{dim}_{\mathbb{K}}(M_i)]t^i.$$
It can be computed from the (minimal) graded free resolution of $M$, see  \cite[Theorem~16.2]{P}.
 $$HS(M,t)=\frac{\sum_{i,j}(-1)^i\beta_{i,\,j}(M)t^j}{(1-t)^n}.$$
 Moreover, there exists a unique polynomial $h_{M}(t)\in \mathbb{Z}[t]$ with $h_{M}(1)\neq0$, which is called the $h$-{\it polynomial} of $M$, such that $HS(M,t)$ can be reduced to the following form by e.g. \cite[Theorem~5.1.4]{V1}:
\begin{equation*}\label{1.1}
  HS(M,t)=\frac{h_{M}(t)}{(1-t)^{\mbox{dim} (M)}}.
\end{equation*}
Here, $\mbox{dim} (M)$ denotes the Krull dimension of $M$.  Fix a monomial order $\prec$ on the ring
$R$. If $I$ is a graded ideal,
then the {\it initial ideal} of $I$, denoted by ${\rm in}_{\prec}(I)$, is the monomial ideal generated
by the initial terms of the polynomials belonging to $I$.  We  need the next result in following sections, see \cite[Corollary 2.5]{OHH02}.

\begin{Lemma}
\label{induced-bound}
Let $G$ be a simple graph and let $H$ be an induced subgraph of $G$. Then we have $\beta_{i,j}(I_H)\leq \beta_{i,j}(I_G)$ for all $i,j\geq 0$.
\end{Lemma}

\section{monomial ideals which have regular quotients}

In this section, we present a generalization of linear quotients by introducing monomial ideals with regular quotients. We also delve into their homological aspects, including projective dimension and regularity, and provide computations of their Betti numbers in certain cases.

Let $R$ denote the standard graded polynomial ring $\mathbb{K}[x_1,\dots,x_n].$ For a monomial $\alpha$, the {\it support} of $\alpha$, denoted by $\mbox{supp}(\alpha)$, is the set of variables $x$ such that $x$ divides $\alpha$. It is known that a sequence of monomials $\alpha_1,\ldots,\alpha_r$ in $R$ is {\it regular} if and only if they have disjoint supports, namely, $\mbox{supp}(\alpha_i)\cap \mbox{supp}(\alpha_j)=\emptyset$ for all $1\leq i\neq j\leq r$. For the explicit definition of a regular sequence we refer to \cite[page 270]{HH}.

\begin{Proposition} \label{general} Suppose the monomial ideal $I$ is minimally generated by $f_1,\ldots,f_m$ such that for $k=2,\ldots,m$, $(f_1,\ldots,f_{k-1}):f_k$ is generated by a regular sequence $\alpha_{k,1}, \ldots, \alpha_{k,r_k}$ with $\deg(\alpha_{k,i})=a_{k,i}$, where $i=1,\ldots,r_k$. Denote by $d_k$ the degree of $f_k$ for $k=1,\ldots,m$.
Then \begin{enumerate}
       \item $HS(I,t)=\frac{1}{(1-t)^n}[t^{d_1}+\sum_{k=2}^m(1-t^{a_{k,1}})\cdots (1-t^{a_{k,r_k}}) t^{d_k}]$;
       \item $\mathrm{reg}(I)\leq \max\{ d_k+a_{k,1}+\cdots+a_{k,r_k}-r_k\:\;  2\leq k\leq m\}$;
       \item $\mathrm{Pdim}(I)\leq \max\{r_2,\ldots,r_m\}$.
     \end{enumerate}

     \begin{proof} Recall the known fact that if $0\rightarrow A\rightarrow B\rightarrow C\rightarrow 0$ is a short exact sequence of finitely generated graded $R$-modules, then $HS(B,t)=HS(A,t)+HS(C,t)$, $\mathrm{Pdim}(B)\leq \max\{\mathrm{Pdim}(A),\mathrm{Pdim}(C)\}$ and $\mathrm{reg}(B)\leq \max\{\mathrm{reg}(A),\mathrm{reg}(C)\}$.

      Put $I_k=(f_1,\ldots,f_{k})$ for $k=1,\ldots,m$ and $L_k=(f_1,\ldots,f_{k-1}):f_k$ for $k=2,\ldots,m$. It follows from the theory of Koszul homology, see e.g. \cite{HH},  that  for $k=2,\ldots,m$, we have \begin{enumerate}
       \item  $HS(R/L_k[-d_k],t)=[(1-t^{a_{k,1}})\cdots (1-t^{a_{k,r_k}})]/(1-t)^n$;
       \item $\mathrm{reg}(R/L_k[-d_k])=d_k+a_{k,1}+\cdots+a_{k,r_k}-r_k$;
       \item $\mathrm{Pdim}(R/L_k[-d_k])=r_k$.
     \end{enumerate}
    Note that $I_1$ is isomorphic to the free module $R[-d_1]$.  It is also clear that $r_2=1$ and $d_1\leq d_2+a_{2,1}-1$.  Hence, the result follows by applying the facts stated before to  the following short exact sequences: $$0\rightarrow I_{k-1}\rightarrow I_k\rightarrow R/L_k[-d_k]\rightarrow 0,$$ where  $k=2,\ldots,m.$ {\hfill $\square$\par}
     \end{proof}
\end{Proposition}

 The upper bounds presented in Proposition~\ref{general} can be attained under a certain condition.  We will see an application of this result in Section 4.

\begin{Corollary} \label{2.5} Let $I$ be as in Proposition~\ref{general}, and let $Q(I)$ and  $P(I)$ denote upper bounds of $\mathrm{reg}(I)$
and $\mathrm{Pdim}(I)$ given in Proposition~\ref{general} respectively. Let $D(I)$ be the degree of the polynomial $(1-t)^nHS(I,t)$.
  Then $\mathrm{reg}(I)=Q(I)$ and  $\mathrm{Pdim}(I)=P(I)$ whenever $P(I)+Q(I)=D(I)$. Moreover, $\beta_{P(I), P(I)+Q(I)}(I)$ is the unique extremal Betti number of $I$ in this case.
\end{Corollary}
\begin{proof} Let $t(I)=\max\{j\:\; \beta_{i,j}(I)\neq 0 \mbox{ for some } i\geq 0\}$. In general, we have $$\mathrm{Pdim}(I)+\mathrm{reg}(I)\geq t(I)\geq D(I).$$
 From the assumption  $P(I)+Q(I)=D(I)$ it follows  that $\mathrm{reg}(I)=Q(I)$ and $\mathrm{Pdim}(I)=P(I)$. Moreover, we have $\beta_{i,t(I)}(I)\neq 0$ for some $i\geq 0$. It is easy to conclude that $i=P(I)$ and the result follows. {\hfill $\square$\par}
\end{proof}

Computing the Betti numbers of monomial ideals with regular quotients is generally challenging. For the remainder of this section, we focus  on ideals where each colon ideal involved is generated in a single degree. The following example illustrates the abundance of such ideals.
\begin{Example}
Let $a_2\leq a_3\leq \cdots\leq a_m$ be an arbitrary increasing sequence of positive integers. Then there exist (square-free) monomials $f_1,\cdots,f_m$ such that
\begin{enumerate}
  \item $\deg(f_1)\leq \cdots \leq \deg(f_m)$;
  \item $(f_1,\cdots,f_{k-1}):f_k$ is generated by a regular sequence of a single degree $a_k$ for $k=2,\cdots,m$.
\end{enumerate}
\begin{proof}
By induction, there exist  monomials $g_1,\cdots,g_{m-1}$ such that $\deg(g_1)\leq \cdots \leq \deg(g_{m-1})$ and $(g_1,\ldots,g_{k-1}):g_k$ is generated by a regular sequence of a single degree $a_k$ for $k=2,\cdots,m-1$. Now, we take  monomials $u_1$ and $u_2$ with $\deg(u_1)=\deg(u_2)=a_m$ and  such that $\mathrm{supp}(u_1)\cap \mathrm{supp}(u_2)=\emptyset$,  $\mathrm{supp}(u_i)\cap \mathrm{supp}(g_j)=\emptyset$ for $i=1,2$ and $j=1,\dots,m-1$.  Then the monomials $f_1,\ldots,f_{m-1},f_m$  is what we want, where $f_1=u_1g_1,\cdots,f_{m-1}=u_1g_{m-1}, f_m=g_1\cdots g_{m-1}u_2$.{\hfill $\square$\par}
\end{proof}
\end{Example}

For the sake of convenience, we give the following definitions.

\begin{Definition}\em
Let $u_1,\ldots,u_r$ be a regular sequence of monomials. We call it a regular sequence of type $r$ and in degree $a$ or just a $(a,r)$-sequence if $\deg(u_1)=\cdots=\deg(u_r)=a$.
\end{Definition}

\begin{Definition}\em
Let $I$ be a monomial ideal of $R$ minimally generated by monomials $f_1,\ldots,f_s$, and $a$ a positive integer. We say it has $a$-regular quotients  if $(f_1,\ldots,f_{i-1}):f_i$ is generated by a regular sequence in degree $a$ for $i=2,\ldots,s$.

\end{Definition}

 The following lemma follows directly from  Koszul homology theory.

\begin{Lemma} \label{K} Let $L$ be an ideal generated by an $(a,r)$-sequence $\alpha_1,\ldots,\alpha_r$. Then $$\beta_{i,j}(R/L)=\left\{
                                                           \begin{array}{ll}
                                                             \begin{pmatrix}r\\i\end{pmatrix}, & \hbox{$j=ai,  i\geq 0$;} \\
                                                             0, & \hbox{otherwise.}
                                                           \end{array}
                                                         \right.
$$
\end{Lemma}

Let $I=(f_1,\ldots,f_m)$ be a monomial ideal  such that $(f_1,\ldots,f_{k-1}):f_k$ is generated by an $(a,r)$-regular sequence  for $k=2,\ldots,m$.   Denote $\deg(f_k)$ by $d_k$ for $k=1,\ldots,m$. By employing the theory of mapping cone, we can derive an upper bound for $\beta_{i,j}(I)$.

Let $\phi: A\rightarrow B$ be a homomorphism of complexes. Recall the mapping cone of $\phi$ is the complex
    $C(\phi)$ with $C(\phi)_i = B_i \oplus A_{i-1}$, and with  the chain map $d_i: C(\phi)_i\rightarrow C(\phi)_{i-1}, d_i(b,a)=(\phi_{i-1}(a)+\partial_i(b),-\partial_{i-1}(a))$  for all $i$. To apply this construction into  the short exact sequences: $$0\rightarrow R/L_k[-d_k]\rightarrow R/I_{k-1}\rightarrow R/I_k\rightarrow 0, \quad \forall\ k=2,\ldots,m$$
where $L_k=(f_1,\ldots,f_{k-1}):f_k$ and $I_k=(f_1,\ldots,f_k)$, we let $K^{(k)}$ be the Koszul complex for the regular sequence which generates  $L_k$,  $F^{(k-1)}$ be the graded free resolution of $R/I_{k-1}$, and $\phi^{(k)}:K^{(k)}[-d_k]\rightarrow F^{(k-1)}$ a graded complex homomorphism lifting $R/L_k[-d_k]\rightarrow R/I_{k-1}$. Then the mapping cone $C(\phi^{(k)})$ yields a graded free resolution of  $R/I_{k}$.
 Thus by iterated mapping cones we obtain step by step
a graded free resolution of $R/I$.

\begin{Theorem} \label{main1} \em Let $I$ be a monomial ideal generated by monomials $f_1,\ldots,f_m$. Denote $\deg(f_i)$ by $d_i$ for $i=1,\ldots,m$. Suppose that the colon ideal $(f_1,\ldots,f_{i-1}):f_i$ is generated by an $(a_i,r_i)$-regular sequence for each $i=2,\ldots,m$. Set $a_1=r_1=0$. Let $A_i$ denote the set $A_i=\{ia_p+d_p\:\; p=1,\ldots,m\}$  for $i\geq 0$, and  $K^i_j$  the set $\{1\leq k\leq m\:\; a_ki+d_k=j\}$ for all $i,j\geq 0$.  Then \begin{enumerate}
                    \item $\beta_{i,j}(I)\leq \left\{
                              \begin{array}{ll}
                                \sum\limits_{k\in K_j^i}  \begin{pmatrix}r_k\\i\end{pmatrix} , & \hbox{$j\in A_i$;} \\
                                0, & \hbox{$j\notin A_i$.}
                              \end{array}
                            \right.$

                    \item  The upper bound given in (1) is achieved if  the following condition holds:
 \begin{equation}\tag{$\clubsuit$}\label{condition} \begin{split} \mbox{ for all }  1\leq i\leq \max\{r_2,\ldots,r_{k}\}+1 \mbox{ and all } k=2,\ldots,m,\\ \mbox{ one has\quad } a_ki+d_k\notin \{a_{\ell}(i-1)+d_{\ell}\:\; 1\leq \ell \leq k-1\}.\quad
 \end{split}\end{equation}
                  \end{enumerate}

\end{Theorem}

\begin{proof}

(1) Based on the preceding discussion, it can be concluded that \begin{equation*}\begin{split} \beta_{i,j}(R/I_m)&\leq \beta_{i-1,j}(R/L_m[-d_m])+\beta_{i,j}(R/I_{m-1})\\ &\leq \cdots\\ &\leq \beta_{i-1,j}(R/L_m[-d_m])+\cdots+\beta_{i-1,j}(R/L_2[-d_2])+\beta_{i,j}(R/(f_1)).\end{split}\end{equation*}

Hence,  \begin{equation*}\begin{split} \beta_{i,j}(I)\leq \beta_{i,j}(R/L_m[-d_m])+\cdots+\beta_{i,j}(R/L_2[-d_2])+\beta_{i,j}(R[-d_1]).                   \end{split}\end{equation*}
The latter is equal to the upper bound given above by Lemma~\ref{K} and so the result follows.

(2) We proceed by induction on $m$. If $m=1$, then $I\cong R[-d_1]$ and so $$\beta_{i,j}(I)=\left\{
                                                                                                    \begin{array}{ll}
                                                                                                      1, & \hbox{$i=0,j=d_1$;} \\
                                                                                                      0, & \hbox{otherwsie.}
                                                                                                    \end{array}
                                                                                                  \right.
$$ Note that $A_i=\{d_1\}$ for all $i\geq 0$, it is easy to check that  the desired formula  is true when $m=1$.

Now, assume $m\geq 2$.   By induction hypothesis, one has $$\beta_{i,j}(I_{m-1})=\left\{
                              \begin{array}{ll}
                                \sum\limits_{k\in \underline{K}^i_j}  \begin{pmatrix}r_k\\i\end{pmatrix}, & \hbox{$j\in B_i$;} \\
                                0, & \hbox{$j\notin B_i$.}
                              \end{array}
                            \right.
$$ Here, $B_i=\{a_pi+d_p\:\; p=1,\ldots,m-1\}$ for all $i\geq 0$ and $\underline{K}^i_j=K^i_j\setminus\{m\}$ for all $i,j\geq 0$. On the other hand,   the following  short exact sequence: $$0\rightarrow I_{m-1}\rightarrow I_m\rightarrow R/L_m[-d_m]\rightarrow 0$$  induces the  long exact sequence:
\begin{equation}
\label{short-exact}
\cdots\rightarrow \mathrm{Tor}_{i+1}^R(R/L_m[-d_m], \mathbb{K})_j\rightarrow \mathrm{Tor}_{i}^R(I_{m-1}, \mathbb{K})_j\rightarrow \mathrm{Tor}_{i}^R(I_{m}, \mathbb{K})_j$$$$\rightarrow \mathrm{Tor}_{i}^R(R/L_m[-d_m], \mathbb{K})_j\rightarrow \mathrm{Tor}_{i-1}^R(I_{m-1}, \mathbb{K})_j\rightarrow \cdots
\end{equation}
Fix $i\geq 0$. If $j\notin A_i$, then $j\notin B_i$ and $j\neq ia_m+d_m$,  and so  $\mathrm{Tor}_{i}^R(I_{m-1}, \mathbb{K})_j=\mathrm{Tor}_{i}^R(R/L[-d_m], \mathbb{K})_j=0$ by induction and by Lemma~\ref{K} respectively. This implies $\mathrm{Tor}_{i}^R(I_{m}, \mathbb{K})_j=0$ and so $\beta_{i,j}(I)=\beta_{i,j}(I_m)=0$. If $j\in A_i$, we consider the following cases.

\noindent {\it Case when $j=ia_m+d_m$}: In this case, $K^i_j=\underline{K}^i_j\cup \{m\}$. In view of condition (\ref{condition}), we have either $j\notin B_{i-1}$ or $i>\max\{r_1,\ldots,r_{m}\}+1$, and thus $\mathrm{Tor}_{i-1}^R(I_{m-1}, \mathbb{K})_j=0$. Note that $\mathrm{Tor}_{i+1}^R(R/L_m[-d_m], \mathbb{K})_j=0$ by Lemma~\ref{K}, it follows that $$\beta_{i,j}(I)=\beta_{i,j}(I_{m-1})+\begin{pmatrix}r_m\\i\end{pmatrix}=\sum\limits_{k\in K^i_j}\begin{pmatrix}r_k\\i\end{pmatrix}.$$

\noindent {\it Case when $j\neq ia_m+d_m$}: In this case, $K^i_j=\underline{K}^i_j$. Since $j\in A_i$, we have either $j\neq (i+1)a_m+d_m$ or $i>r_{m}+1$, according to  (\ref{condition}). It follows that $\mathrm{Tor}_{\ell}^R(R/L_m[-d_m], \mathbb{K})_j=0$ for $\ell\in \{i,i+1\}$ by Lemma~\ref{K}. Hence, $$\beta_{i,j}(I)=\beta_{i,j}(I_{m-1})=\sum\limits_{k\in \underline{K}^i_j}\begin{pmatrix}r_k\\i\end{pmatrix}=\sum\limits_{k\in K^i_j}\begin{pmatrix}r_k\\i\end{pmatrix}.$$ This completes the proof.{\hfill $\square$\par}
\end{proof}

\begin{Remark} \em Note that if $j\notin A_i$, then $K^i_j=\emptyset$, and in particular,  $\sum\limits_{k\in K_j^i}  \begin{pmatrix}r_k\\i\end{pmatrix}=0$. Hence, the conclusion in Theorem~\ref{main1}.(2) may be  reformulated  as:
$$\beta_{i,j}(I)= \sum\limits_{k\in K_j^i}  \begin{pmatrix}r_k\\i\end{pmatrix}.$$
\end{Remark}

\begin{Remark} \label{r=1} \em If $r_2=r_3=\cdots=r_m=1$, then condition (\ref{condition}) can be removed from Theorem~\ref{main1}. Since   $\beta_{i,j}(I)=0$ for all $i, j$ with $i\geq 2$, it suffices to prove the formula in Theorem~\ref{main1}.(2) holds both when  $i=0$ and when $i=1$. We show this again by induction on $m$.  The case when $i=0$ follows by noting that $$\beta_{0,j}(I_m)-\beta_{0,j}(I_{m-1})=\left\{
                                                                                                            \begin{array}{ll}
                                                                                                              1, & \hbox{$j=d_m$;} \\
                                                                                                              0, & \hbox{otherwise.}
                                                                                                            \end{array}
                                                                                                          \right.
 $$
On the other hand,  since $\mathrm{Tor}_2^R(R/L_m[-d_m],\mathbb{K})_j=0$  for any $j\geq 0$, in view of the long exact sequence  (\ref{short-exact}),
 we have \begin{equation*}\begin{split}\beta_{1,j}(I_m)=\beta_{1,j}(I_{m-1})+\beta_{1,j}(R/L_m[-d_m]) +\beta_{0,j}(I_m)-\beta_{0,j}(I_{m-1})-\beta_{0,j}(R/L_m[-d_m]).\end{split}\end{equation*} From this it follows that $\beta_{1,j}(I_m)=\beta_{1,j}(I_{m-1})+\beta_{1,j}(R/L_m[-d_m])$, completing the proof.
\end{Remark}

Let us see two special cases of Theorem~\ref{main1}, which we are interested in.

\begin{Corollary} \label{main11} Let $I$ be a monomial ideal generated by monomials $f_1,\ldots,f_m$. Denote $\deg(f_i)$ by $d_i$ for $i=1,\ldots,m$. Suppose  the colon ideal $(f_1,\ldots,f_{k-1}):f_k$ is generated by an $(a_k,r_k)$-regular sequence for each $k=2,\ldots,m$. Moreover, if one of the following conditions is satisfied:

\begin{enumerate}
  \item  $a_2=\ldots=a_m=a$ and  $d_{\ell}-d_{k}\neq a$ for $k\geq 2$ and $1\leq\ell \leq  k-1$;

  \item  $a_1\leq a_2\leq \cdots\leq a_m$ and $d_1\leq d_2\leq \cdots\leq d_m$;
\end{enumerate}

then the conclusion in Theorem~\ref{main1}.(2) holds.

\end{Corollary}

\begin{proof} This is because if one of the conditions mentioned above is satisfied then condition (\ref{condition}) in Theorem~\ref{main1} holds true.{\hfill $\square$\par}
\end{proof}

The inspiration for the following result originates from  \cite{SJZ10}.  It suggests if $I$ is a monomial ideal with $a$-regular quotients,  we may rearrange the order of the generators of $I$ to make their degrees more closely approximate a monotonically increasing pattern, while preserving the quotient ideals unchanged. If $a=1$, we recover \cite[Lemma 2.1 and Proposition 2.2]{SJZ10}.

\begin{Lemma} \label{re} Let $I$ be a monomial ideal having regular quotients with respect to the order $f_1,\ldots,f_m$ of minimal generators of $I$ such that all the quotient ideals are generated in the same degree $a$. Then there is a permutation $\sigma$ on $[n]$ with $\sigma(1)=1$ such that  for any  $i\in [n]$, one has $$(f_{\sigma(1)},\ldots,f_{\sigma(i-1)}):f_{\sigma(i)}=(f_1,\ldots,f_{\sigma(i)-1}):f_{\sigma(i)}.$$  Moreover, $\deg(f_{\sigma(i)})\leq \deg(f_{\sigma(i+1)})+a-1$ for $i=1,\ldots,m-1$.
\end{Lemma}

\begin{proof}  By induction, we may assume  $\deg(f_i)\leq \deg(f_{i+1})+(a-1)$ for $i=1,\ldots,m-2$. If $\deg(f_{m-1})\leq \deg(f_m)+a-1$, we are done. Suppose now $\deg(f_{m-1})\geq \deg(f_m)+a$. Then, there is an integer $1\leq j\leq m-1$ such that $\deg(f_j)\leq \deg(f_m)+a-1$ and $\deg(f_p)\geq \deg(f_m)+a$ for $p=j+1,\ldots,m-1$.  We only  need to prove $f_1,\ldots, f_j,f_m,f_{j+1},\ldots, f_{m-1}$ is the order we require. In the following, we denote $u:v$ as the monomial $u/(u,v)$, where $(u,v)$ represents the greatest common divisor of $u$ and $v$.

Since   $\deg(f_p:f_m)\geq a+1$ for $p=j+1,\ldots,m-1$,  it follows that $(f_1,\ldots,f_j):f_m=(f_1,\ldots,f_{m-1}):f_m$. It remains to show that for all $p$ with $ j+1\leq p\leq m-1$, one has $$(f_1,\ldots,f_j, f_m, f_{j+1},\ldots,f_{p-1}):f_p=(f_1,\ldots,f_j, f_{j+1},\ldots,f_{p-1}):f_p.$$
This is equivalent to showing that $f_m:f_p\in (f_1,\ldots,f_j, f_{j+1},\ldots,f_{p-1}):f_p.$ To do this, we first note that there is an integer $1\leq k\leq j$
such that $\deg(f_k:f_m)=a$ and $f_k:f_m$ divides $f_p:f_m$. Put $\alpha=f_k:f_m.$  There are two cases to consider:

  The first case is when $\deg(f_k: f_p)>a$. In this case,  there is $1\leq \ell \leq j$ such that $\deg(f_{\ell}:f_p)=a$ and $f_{\ell}:f_p$ divides $f_k:f_p$. Put $\beta=f_{\ell}:f_p$. Then, for any variable $x\in \mathrm{supp} (\alpha)$, we have
$$\deg_x(f_k)=\deg_x(f_m)+\deg_{x}\alpha,  \quad \deg_x(f_p)\geq \deg_x(f_m)+\deg_{x}\alpha.$$
This implies that $\deg_x(f_k:f_p)=0$, and so $x\notin \mathrm{supp}(\beta)$.  Thus, $$\mathrm{supp}(\alpha)\cap \mathrm{supp}(\beta)=\emptyset.$$

Take any variable $y\in  \mathrm{supp}(\beta)$. Then, we have $\deg_{y}(f_m) \geq \deg_y(f_k)$, since  $\mathrm{supp}(\alpha)\cap \mathrm{supp}(\beta)=\emptyset$ as well as $\alpha=f_k:f_m.$  On the other hand, since $\beta$ divides $f_k:f_p$, it follows that  $\deg_y(f_k)\geq \deg_y(f_p)+\deg_y(\beta)$.  Hence, $\deg_{y}(f_m)\geq \deg_y(f_p)+\deg_y(\beta)$, and this implies $\beta$ divides $f_m:f_p$. In particular,  $f_m:f_p$ belongs to $(f_1,\ldots,f_j, f_{j+1},\ldots,f_{p-1}):f_p$,  as  desired.

The second case is when $\deg(f_k: f_p)=a$. In this case, we put $\beta=f_k: f_p$. Now, the remaining proof  is similar to the one in the first case and so we omit it.{\hfill $\square$\par}
\end{proof}

\begin{Example}
\em Consider the monomial ideal $$I=(f_1=x_1x_2x_3x_4x_5x_6, f_2=x_1x_2x_3x_4x_7x_8x_9, f_3=x_3x_4x_5x_6x_7, f_4=x_5x_6x_7x_8).$$
It has $2$-regular quotients with respect to the given order, where the quotients, one by one, are $(x_5x_6), (x_1x_2), (x_3x_4)$. It also  has 2-regular quotients with respect to the order $f_1,f_3,f_4,f_2$, where the quotient ideals, in sequence, are $(x_1x_2), (x_3x_4), (x_5x_6)$, as claimed by Lemma~\ref{re}.  Since every quotient ideal is generated by a single element, we obtain the following Betti table of $I$ by Remark~\ref{r=1}:
\[
\begin{array}{c|cc}
    & 0 & 1 \\
\hline
  4 & 1 & 0 \\
  5 & 1 & 1 \\
  6 & 1 & 1 \\
  7 & 1 & 0 \\
  8 & 0 & 1
\end{array}.
\]

It is interesting to see that $I$ has regular quotients with respect to many other orders. For examples:

With respect to the degree increasing order $f_4,f_3,f_1,f_2$,  $I$ has $(1,1,2)$-regular quotients. More precisely, $f_4:f_3=(x_8), (f_4,f_3):f_1=(x_7)$ and $(f_4,f_3,f_1):f_2=(x_5x_6)$.

With respect to the order $f_1,f_4,f_3,f_2$, the ideal $I$ has regular quotients too:  $(f_4):f_1=(x_7x_8)$, $(f_4,f_1):f_3=(x_1x_2,x_8)$, and $(f_4,f_1,f_3):f_2=(x_5x_6)$. This shows  the upper bounds given in Proposition~\ref{general} may not be achieved.

With respect to the degree decreasing  order $f_2,f_1,f_3,f_4$, the ideal $I$ has regular quotients again: $(f_2):f_1=(x_7x_8x_9)$. $(f_2,f_1):f_3=(x_1x_2)$, $(f_2,f_1,f_3):f_4=(x_3x_4)$.
\end{Example}

\section{A bridge theorem}

In this section, we establish a condition that guarantees the equality of Betti numbers between a toric ideal of an edge ring and its initial ideal. This condition will be utilized in subsequent sections.

\begin{Lemma} \label{pure} Let $I$ be a homogeneous ideal of $R$ such that ${\mathrm{in}}_{\prec}(I)$ has a pure resolution for some monomial order $\prec$. Then $\beta_{i,j}(I)=\beta_{i,j}(\mathrm{in}_{\prec}(I))$ for all $i,j\geq 0$.
\begin{proof}
It was shown in \cite[Corollary~22.13]{P} that if ${\rm{in}}_{\prec}(I)$ has a linear resolution, then $\beta_{i,j}(I)=\beta_{i,j}({\rm{in}}_{\prec}(I))$ for all $i,j\geq 0$. The same proof  applies to the case of pure resolution.{\hfill $\square$\par}
\end{proof}
\end{Lemma}

\begin{Theorem}
\label{bridge}
Let $G$ be a simple graph and $s$ a positive integer. Assume that $\mathcal{PC}(G)$  can be arranged  as  $\{C_1, \ldots, C_r\}$ such that the following conditions are satisfied:
 \begin{enumerate}
\item With respect to  some monomial order $\prec$, the monomial ideal $\mathrm{in}_{\prec}(I_G)$ has $s$-regular quotients with respect to the order $u_1,u_2,\ldots,u_r$;
Here $u_i=\mathrm{in}_{\prec}(f_{C_i})$ for $i=1,\ldots,r$.
\item There are pairwise distinct  positive integers $d_1,\ldots,d_m$   such that \begin{equation*}\begin{split}\deg(u_1)=\cdots=& \deg(u_{k_1})=d_1, \deg(u_{k_1+1})=\cdots=\deg(u_{k_2})=d_2,\\ &\cdots, \cdots, \deg(u_{k_{m-1}+1})=\cdots=\deg(u_{k_m})=d_m\end{split}\end{equation*} for some $k_i$'s with $k_m=r$ and such that there exists an induced subgraph $G_i$ of $G$ satisfying $$\mathcal{PC}(G_i)=\{C_{\ell}\:\; 1\leq \ell\leq k_i\}$$ for each $i=1,2,\ldots,m$;
      \item $d_k-d_{\ell}\neq s$ for any pair $k,\ell$ with $1\leq k<\ell\leq m$.
\end{enumerate}
Then $\beta_{i,j}(I_G)=\beta_{i,j}(\mathrm{in}_{\prec}(I_G))$ for all $i,j\geq 0$.
\end{Theorem}
\begin{proof}  Observe that $\mathrm{in}_{\prec}(I_{G_i})=(u_1,\ldots, u_{k_i})$ for $i=1,2,\ldots,m$ and that we may regard $G_i$ as an induced subgraph of $G_{i+1}$ for $i=1,\ldots,m-1$. In the sequel,   we use induction on $m$. If $m=1$, then $\beta_{i,j}(\mathrm{in_{<}}(I_G))=0$ when $j\neq si+d_1$ by Theorem~\ref{main1}. It follows that  $\mathrm{in}_{\prec}(I_G)$ has a pure resolution, and so we are done by Lemma~\ref{pure}.

Suppose  $m>1$. Note that $G=G_m$.  By induction hypothesis, we have $\beta_{i,j}(I_{G_{m-1}})=\beta_{i,j}(\mathrm{in}_{\prec}(I_{G_{m-1}}))$ for all $i,j\geq 0$.  Additionally,  by  Theorem~\ref{main1}.1 as well as \cite[Corollary 3.3.3]{HH}, it follows that $\beta_{i,j}(I_G)=0$ for all $j\notin \{is+d_k\:\; k=1,\ldots,m\}$.

Fix $i\geq 0$ and let $j\in \{is+d_k\:\; k=1,\ldots,m-1\}$. Since $d_1,\ldots,d_m$ are pairwise distinct, $\beta_{i,j}(\mathrm{in}_{\prec}(I_{G_{m-1}}))=\beta_{i,j}(\mathrm{in}_{\prec}(I_{G_m}))$ by Corollary~\ref{main11}. Based on this and by utilizing Lemma~\ref{induced-bound}, it can be inferred that
$$\beta_{i,j}(I_G)\geq \beta_{i,j}(I_{G_{m-1}})=\beta_{i,j}(\mathrm{in}_{\prec}(I_{G_{m-1}}))=\beta_{i,j}(\mathrm{in}_{\prec}(I_{G_m}))\geq \beta_{i,j}(I_G).$$ This implies that
\begin{equation} \label{m-1} \beta_{i,j}(I_G)=\beta_{i,j}(\mathrm{in}_{\prec}(I_{G})) \mbox{  for all } i\geq 0 \mbox{ and } j\in \{is+d_k\:\; k=1,\ldots,m-1\}.
\end{equation}

It remains to show that $\beta_{i,j}(I_G)=\beta_{i,j}(\mathrm{in}_{\prec}(I_{G}))$  whenever $j=is+d_m$. To do this, we fix $i\geq 0$ and let  $j=is+d_m$. By putting $A:=\{k\geq 0\:\; ks+d_t=j \mbox{\ for some\ } 1\leq t\leq m-1\}$, we have the  coefficient of $t^j$ in $(1-t)^nHS(I_G, t)$ is equal to $$(-1)^i\beta_{i,j}(I_G)+\sum_{k\in A}(-1)^k\beta_{k,j}(I_G),$$
and the  coefficient of $t^j$ in $(1-t)^nHS(\mathrm{in}_{\prec}(I_G), t)$ is equal to $$(-1)^i\beta_{i,j}(\mathrm{in}_{\prec}(I_G))+\sum_{k\in A}(-1)^k\beta_{k,j}(\mathrm{in}_{\prec}(I_G)).$$
     From (\ref{m-1}), it follows that   $$\sum_{k\in A}(-1)^k\beta_{k,j}(I_G)=\sum_{k\in A}(-1)^k\beta_{k,j}(\mathrm{in}_{\prec}(I_G)).$$ By combining these factors, we obtain $\beta_{i,j}(I_G)=\beta_{i,j}(\mathrm{in}_{\prec}(I_G))$, as desired.{\hfill $\square$\par}
\end{proof}
\section{Bipartite graphs $B_{\underline{\ell},h}$ and their homological invariants}
Let $h\geq 2$ be an integer, and  $\underline{\ell}=(\ell_1, \ldots, \ell_h)$  a vector in $\mathbb{Z}_{>0}^h$. In this section, we introduce a special family of bipartite graphs denoted by $B_{\underline{\ell},h}$. We then compute the regularity and projective dimension of the edge ring $\mathbb{K}[B_{\underline{\ell},h}]$ in a general setting. Additionally, we provide an explicit formula for the Betti numbers of the toric ideal of $B_{\underline{\ell},h}$ in the case when $\ell_2=\cdots=\ell_h$.

\begin{Definition}\em Let $h\geq 2$ be an integer and let  $\underline{\ell}=(\ell_1, \ldots, \ell_h)$ be a vector in $\mathbb{Z}_{>0}^h$. We define $B_{\underline{\ell},h}$ to be a bipartite graph consisting of $h$ paths of lengths $2\ell_1,\ldots,2\ell_h$ respectively, all originating from a shared vertex and terminating at another common vertex, see Fig. 1. More precisely,
$B_{\underline{\ell},h}$ is the bipartite graph  with  vertex set
$$V_{\underline{\ell},h}=\{y_1,y_2\}\cup \{x_{i,j} \:\; 1\leq i\leq 2\ell_k-1, 1\leq k,j\leq  h \}$$
and with edge    set
\begin{equation*}
  \begin{split}
     E_{\underline{\ell},h} & =\{\{y_1,x_{1,1}\},\ldots,\{y_1,x_{1,h}\}\}\cup \{\{x_{i,j},x_{i+1,j}~|~1\leq i \leq 2\ell_k-2, 1\leq k,j \leq h\}\} \\
       & \quad \cup \{\{y_2,x_{2\ell_1,1}\},\ldots,\{y_2,x_{2\ell_h,h}\}\}.
   \end{split}
\end{equation*}
\end{Definition}
\noindent We label the edges of $B_{\underline{\ell},h}$ as follows: $e_{1,j}=\{y_1,x_{1,j}\},e_{i,j}=\{x_{i-1,j},x_{i,j}\},e_{2\ell,j}=\{y_2,x_{2\ell,j}\}$, for $i\in \{2,\ldots,2\ell_k-1\}$ and $k,j\in \{1,\ldots,h\}.$

\begin{figure}[htbp]%
  \label{Figure1}  \centering

        \begin{tikzpicture}[thick, scale=1.0, every node/.style={scale=0.8}]]
        \fill (0,0) circle (.07);
        \fill (10,0) circle (.07);
        \fill (2,3) circle (.07);
        \fill (4,3) circle (.07);
        \fill (6,3) circle (.07);
        \fill (8,3) circle (.07);
        \fill (2,2) circle (.07);
        \fill (4,2) circle (.07);
        \fill (6,2) circle (.07);
        \fill (8,2) circle (.07);
        \fill (2,1) circle (.07);
        \fill (4,1) circle (.07);
        \fill (6,1) circle (.07);
        \fill (8,1) circle (.07);
        \fill (2,-1) circle (.07);
        \fill (4,-1) circle (.07);
        \fill (6,-1) circle (.07);
        \fill (8,-1) circle (.07);
        \fill (8,-2) circle (.07);
        \fill (2,-2) circle (.07);
        \fill (4,-2) circle (.07);
        \fill (6,-2) circle (.07);
        \fill (8,-2) circle (.07);
        \fill (2,-3) circle (.07);
        \fill (4,-3) circle (.07);
        \fill (6,-3) circle (.07);
        \fill (8,-3) circle (.07);
        \draw[thick] (0,0) -- (2,3);
        \draw[thick] (2,3) -- (4,3);
        \draw[thick,dashed] (4,3) -- (6,3);
        \draw[thick] (6,3) -- (8,3);
        \draw[thick] (8,3) -- (10,0);
        \draw[thick] (0,0) -- (2,2);
        \draw[thick] (2,2) -- (4,2);
        \draw[thick,dashed] (4,2) -- (6,2);
        \draw[thick] (6,2) -- (8,2);
        \draw[thick] (8,2) -- (10,0);
        \draw[thick] (0,0) -- (2,1);
        \draw[thick] (2,1) -- (4,1);
        \draw[thick,dashed] (4,1) -- (6,1);
        \draw[thick] (6,1) -- (8,1);
        \draw[thick] (8,1) -- (10,0);
        \draw[thick] (0,0) -- (2,-1);
        \draw[thick] (2,-1) -- (4,-1);
        \draw[thick,dashed] (4,-1) -- (6,-1);
        \draw[thick] (6,-1) -- (8,-1);
        \draw[thick] (8,-1) -- (10,0);
        \draw[thick] (0,0) -- (2,-2);;
        \draw[thick] (2,-2) -- (4,-2);
        \draw[thick,dashed] (4,-2) -- (6,-2);
        \draw[thick] (6,-2) -- (8,-2);
        \draw[thick] (8,-2) -- (10,0);
        \draw[thick] (0,0) -- (2,-3);
        \draw[thick] (2,-3) -- (4,-3);
        \draw[thick,dashed] (4,-3) -- (6,-3);
        \draw[thick] (6,-3) -- (8,-3);
        \draw[thick] (8,-3) -- (10,0);
        \draw[very thick,dotted] (3,0.4) -- (3,-0.4);
        \draw[very thick,dotted] (7,0.4) -- (7,-0.4);
        \draw (1.3,1.8) node[anchor=south east]{$e_{1,1}$};
        \draw (1.5,1) node[anchor=south east]{$e_{1,2}$};
        \draw (1.5,0.5) node[anchor=south east]{$e_{1,3}$};
        \draw (1.8,-0.5) node[anchor=south east]{$e_{1,h-2}$};
        \draw (1.8,-1.3) node[anchor=south east]{$e_{1,h-1}$};
        \draw (1.3,-2) node[anchor=south east]{$e_{1,h}$};
        \draw (3,3) node[anchor=south]{$e_{2,1}$};
        \draw (3,2) node[anchor=south]{$e_{2,2}$};
        \draw (3,1) node[anchor=south]{$e_{2,3}$};
        \draw (3,-1) node[anchor=north]{$e_{2,h-2}$};
        \draw (3,-2) node[anchor=north]{$e_{2,h-1}$};
        \draw (3,-3) node[anchor=north]{$e_{2,h}$};
        \draw (7,3) node[anchor=south]{$e_{2\ell_1-1,1}$};
        \draw (7,2) node[anchor=south]{$e_{2\ell_2-1,2}$};
        \draw (7,1) node[anchor=south]{$e_{2\ell_3-1,3}$};
        \draw (7,-1) node[anchor=north]{$e_{2\ell_{h-2}-1,h-2}$};
        \draw (7,-2) node[anchor=north]{$e_{2\ell_{h-1}-1,h-1}$};
        \draw (7,-3) node[anchor=north]{$e_{2\ell_{h}-1,h}$};
        \draw (8.7,1.8) node[anchor=south west]{$e_{2\ell_1,1}$};
        \draw (8.2,1.3) node[anchor=south west]{$e_{2\ell_2,2}$};
        \draw (8.2,0.7) node[anchor=south west]{$e_{2\ell_3,3}$};
        \draw (8.2,-0.5) node[anchor=south west]{$e_{2\ell_{h-2},h-2}$};
        \draw (8.2,-1.3) node[anchor=south west]{$e_{2\ell_{h-1},h-1}$};
        \draw (8.7,-2.1) node[anchor=south west]{$e_{2\ell_h,h}$};
        \draw (0,0) node[anchor=east]{$y_1$};
        \draw (10,0) node[anchor=west]{$y_2$};
        \draw (5,-3.5) node[anchor=north]{Fig. 1: The graph $B_{\underline{\ell},h}$};

        \end{tikzpicture}

\end{figure}

Let $\prec$ denote the pure  lexicographic monomial order on the polynomial ring $R:=\mathbb{K}[e_{i,j}\:\; 1\leq i\leq 2\ell_j, 1\leq j\leq h]$,
where the variables are ordered in the following way:
\begin{equation*} e_{i_1,j_1}\succ  e_{i_2,j_2}\Longleftrightarrow \left\{
                                                                     \begin{array}{ll}
                                                                    \mbox{ either }  i_1 \mbox{ is odd  and } i_2 \mbox{ is even}  \\
                                                                     \mbox{or } i_1-i_2=0\ (\mod 2) \mbox { and } i_1<i_2\\
                                                                     \mbox{or } i_1=i_2 \mbox { and } j_1<j_2.
                                                                     \end{array}
                                                                   \right.
\end{equation*}

For all $1\leq i<j\leq h$, we set
 $$u_{i,j}=e_{1,i}e_{3,i}\cdots e_{2\ell_i-1,i}e_{2,j}e_{4,j}\cdots e_{2\ell_j,j}$$
and $$v_{i,j}=e_{1,j}e_{3,j}\cdots e_{2\ell_j-1,j}e_{2,i}e_{4,i}\cdots e_{2\ell_i,i}.$$
Note that both $u_{i,j}$ and $v_{i,j}$ are monomials of degree $\ell_i+\ell_j$, and $v_{i,j}\prec u_{i,j}$ with respect to the order given above.
\begin{Lemma}
 The universal Gr\"{o}bner basis of $I_{B_{\underline{\ell},h}}$ is given by
$$\mathcal{G}_{\ell,h}=\{u_{i,j}-v_{i,j}\:\; 1\leq i<j \leq h\}.$$
In particular, the minimal generating set of $\mathrm{in}_{\prec}(I_{B_{\underline{\ell},h}})$  is
 $$\{u_{i,j}\:\; 1\leq i<j\leq h\}.$$
\begin{proof}

According to \cite[Lemma~3.2]{OH} (or see Proposition~\ref{primitive}), we have the set of primitive even closed walks of $B_{\underline{\ell},h}$ is
 $$\mathcal{PC}(B_{\underline{\ell},h})=\{W_{i,j}=e_{1,i},e_{2,i},\ldots,e_{2\ell_i,i},e_{2\ell_j,j},\ldots,e_{2,j},e_{1,j}\:\; 1\leq i<j\leq h \}.$$
 Note that $f_{W_{i,j}}=u_{i,j}-v_{i,j}$ for all $i,j$, it follows that $\{u_{i,j}-v_{i,j}\:\; 1\leq i<j \leq h\}$ is the universal  Gr\"{o}bner basis of $I_{B_{\underline{\ell},h}}$ by e.g. \cite[Proposition~5.19]{EH}. Now, the last assertion follows since $v_{i,j}\prec u_{i,j}$ for all $i,j$ with respect to the given order.{\hfill $\square$\par}
\end{proof}
\end{Lemma}

It is clear that \begin{equation}\label{4.1}\begin{split}& u_{h-1,h}\prec u_{h-2,h}\prec u_{h-2,h-1}\prec \cdots \prec u_{i,h}\prec \cdots \prec u_{i,i+1}\\ &\prec  u_{i-1,h}\prec \cdots \prec  u_{i-1,i} \prec \cdots\prec u_{1,h}\prec \cdots \prec u_{1,2}.
\end{split}
\end{equation}

We now show that $\mathrm{in}_{\prec}(I_{B_{\underline{\ell},h}})$ has regular quotients with respect to the above order. To state the result more simply, we label the monomials in (\ref{4.1}) one by one: $$m_1=u_{h-1,h}, \quad m_2=u_{h-2,h},\quad  \ldots, \quad m_{\frac{h(h-1)}{2}}=u_{1,2}.$$

\begin{Proposition}
\label{main3} The initial ideal  $\mathrm{in}_{\prec}(I_{B_{\underline{\ell},h}})$ has regular quotients with respect to the above order. More precisely, the colon ideal $(m_1,\ldots, m_{k-1}):m_k$ is generated by a  regular sequence of length $r_k$ for $k=2,\ldots,\frac{h(h-1)}{2}$, where
$$(r_2,\ldots,r_{\frac{h(h-1)}{2}})=(\underbrace{{1,1}}_{2},\underbrace{2,2,2}_{3},\ldots,\underbrace{h-2,h-2,\ldots,h-2}_{h-1}).$$
\begin{proof} By putting  $Y_{1,i}=e_{1,i}e_{3,i}\cdots e_{2\ell_i-1,i}$ and  $Y_{2,j}=e_{2,j}e_{4,j}\cdots e_{2\ell_j,j}$, we may write $$u_{i,j}=Y_{1,i}Y_{2,j}$$ for $1\leq i<j\leq h.$  Fix $2\leq k\leq \frac{h(h-1)}{2}$, then there exist unique $i,j$ with $1\leq i<j\leq h$ such that $m_{k-1}=u_{i,j}$. Let $L_k$ denote the colon ideal $(m_1,\ldots, m_{k-1}):m_k$.

If $j-i\geq 3$, then  $$L_k=(Y_{1,h-1}Y_{2,h}, \ldots, Y_{1,i}Y_{2,j}): Y_{1,i}Y_{2,j-1}=(Y_{1,j-2},\ldots,Y_{1,i+1}, Y_{2,j} \ldots, Y_{2,h})$$ and so it is generated by a regular sequence of length $h-i-1$.

If $j-i= 2$, then  $$L_k=(Y_{1,h-1}Y_{2,h}, \ldots, Y_{1,i}Y_{2,j}): Y_{1,i}Y_{2,i+1}=(Y_{2,i+2} \ldots, Y_{2,h})$$ and so it is generated by a regular sequence of length $h-i-1$.

If $j-i=1$, then $i\geq 2$, and   $$L_k=(Y_{1,h-1}Y_{2,h}, \ldots, Y_{1,i}Y_{2,j}): Y_{1,i-1}Y_{2,h}=(Y_{1,h-1}, Y_{1,h-2}, \ldots, Y_{1,i})$$ and so it is generated by a regular sequence of length $h-i$.

The  facts above  show that  if $m_{k-1}=u_{i,j}$ then $$r_k=\left\{
                                                                          \begin{array}{ll}
                                                                            h-i-1, & \hbox{$j-i\geq 2$  and $i\geq 1$;} \\
                                                                            h-i, & \hbox{$j-i= 1$ and $i\geq 2$.}
                                                                          \end{array}
                                                                        \right.
$$
From this  the last assertion of this result follows.{\hfill $\square$\par}
\end{proof}
\end{Proposition}

\begin{Theorem} \label{CM}
Let $I$ denote the toric ideal $I_{B_{\underline{\ell},h}}$ and $J$ its initial ideal  $\mathrm{in}_{\prec}(I)$. Then  both $R/I$ and $R/J$ are Cohen-Macaulay. Moreover,
\begin{enumerate}
  \item $\mathrm{Pdim}(I)=\mathrm{Pdim}(J)=h-2$;
  \item $\mathrm{reg}(I)=\mathrm{reg}(J)=\ell_1+\cdots+\ell_h-h+2$;
   \item $\deg h_{R/I}(t)=\deg h_{R/J}(t) =\ell_1+\cdots+\ell_h-h+1$;
  \item Denote $p=h-2$ and $q=\ell_1+\cdots+\ell_h$, then $\beta_{p,q}(I)=\beta_{p,q}(J)=h-1$, and they are the only  extreme Betti-numbers of $I$ and $J$ respectively.

    \end{enumerate}
\end{Theorem}

\begin{proof}  Let $n$ denote $2(\ell_1+\cdots+\ell_h)$,  the number of variables of polynomial ring $\mathbb{K}[e_{i,j}\:\; 1\leq i\leq 2\ell_j, 1\leq j\leq h]$, and  let $L_k$ denote the ideal $(m_1,\ldots, m_{k-1}):m_k$ for $k=2,\ldots,\frac{h(h-1)}{2}$. If $m_{k-1}=u_{i,j}$, then, by  Proposition~\ref{main3}, we have \begin{equation*}\begin{split}(1-t)^n &HS(R/L_k[-m_k],t)\\=&\left\{
                                                                                                                 \begin{array}{ll}
                                                                                                                 (1- t^{\ell_i})\cdots (1-t^{\ell_h}) \frac{t^{\ell_i}t^{\ell_{j-1}} }{(1-t^{\ell_i})(1-t^{\ell_{j-1}})}, & \hbox{$j-i\geq 2$  and $i\geq 1$;} \\& \hbox{}\\
                                                                                                                  (1-t^{\ell_{i-1}})\cdots (1-t^{\ell_h})\frac{t^{\ell_{i-1}}t^{\ell_h}}{(1-t^{\ell_{i-1}})(1-t^{\ell_h})}, & \hbox{$j-i= 1$ and $i\geq 2$.}
                                                                                                                 \end{array}
                                                                                                               \right.
\end{split}\end{equation*}  and so the initial term of $(1-t)^nHS(R/L_k[-m_k],t)$ is as follows: \begin{equation*}\begin{split}\left\{
                                                             \begin{array}{ll}
                                                               (-1)^{h-i-1}t^{\ell_i+\cdots+\ell_h}, & \hbox{$j-i\geq 2$  and $i\geq 1$;} \\
                                                               (-1)^{h-i}t^{\ell_{i-1}+\cdots+\ell_h}, & \hbox{$j-i= 1$ and $i\geq 2$.}
                                                             \end{array}
                                                           \right.\end{split}\end{equation*}
 Hence, the initial term of $H(J,t)(1-t)^n$ is equivalent to $$(-1)^{h-2}(h-1)t^{\ell_1+\cdots+\ell_h}.$$
By adopting the notions presented in Corollary~\ref{2.5}, we have
 $D(J)=\ell_1+\cdots+\ell_h$. It  also follows from  Proposition~\ref{main3} that $Q(J)=\ell_1+\cdots+\ell_h-h+2$ and $P(J)=h-2$. This implies $D(J)=P(J)+Q(J)$, and thus $\mathrm{Pdim}(J)=h-2$ and $\mathrm{reg}(J)=\ell_1+\cdots+\ell_h-h+2$ by Corollary~\ref{2.5}.

 In view of \cite[Corollary 10.1.21]{V1}, we have $$\dim (R/J)=\dim \mathbb{K}[B_{\underline{\ell},h}]=2(\ell_1+\cdots+\ell_h)-h+1=n-h+1,$$
and $$\mathrm{depth}(R/J)=n-\mathrm{Pdim}(R/J)=n-h+1.$$
It follows that both $R/I$ and $R/J$ are Cohen-Macaulay. From this, we have $\mathrm{Pdim}(I)=\mathrm{Pdim}(J)$ and $\mathrm{reg}(I)=\mathrm{reg}(J)$. The last statements  are also clear now.{\hfill $\square$\par}
\end{proof}
\begin{Corollary} Let $I$ and $J$ be the same  as in Theorem~\ref{CM}. Then the following are equivalent:\begin{enumerate}
                                                                                  \item $R/I$ is a complete intersection;
                                                                                  \item $R/J$ is a complete intersection;
                                                                                  \item $R/I$ is  Gorenstein;
                                                                                  \item $R/J$ is Gorenstein;
                                                                                  \item $h=2$.
                                                                                \end{enumerate}
\end{Corollary}

\begin{proof} If $h=2$, then $B_{\underline{\ell},h}$ is an even cycle and so  $I$ and $J$ are both generated by a  polynomial. This implies both $R/I$ and $R/J$ are complete intersections. Conversely, if either $R/I$ or $R/J$ is Gorenstein, then either $\beta_p(I)=1$ or  $\beta_p(J)=1$. It follows that $h-1=\beta_{p,q}(I)=\beta_{p,q}(J)\leq 1$, and hence $h=2$.{\hfill $\square$\par}
\end{proof}

\begin{Corollary} Given a vector $(p,r)\in \mathbb{Z}^2$ with $p\geq 0$ and $r\geq 2$, there exists a bipartite graph $G$ such that $\mathrm{pdim}(I_G)=p$ and $\mathrm{reg}(I_G)=r$.
\end{Corollary}
\begin{proof} Put $h=p+2$ and $\underline{\ell}=(\underbrace{1,\ldots,1}_{h-1},r-1)$. Then the bipartite graph $B_{\underline{\ell},h}$ is what we require.{\hfill $\square$\par}
\end{proof}

 In the final part of this section, we  compute their Betti numbers of $\mathbb{K}[B_{\underline{\ell},h}]$ explicitly in the case when $\ell_2=\cdots=\ell_{h}$.
\begin{Theorem}
 Let $B_{\underline{\ell},h}$ be the bipartite graph defined above with $\underline{\ell}=(\ell_1,\ldots,\ell_h)$.

\begin{enumerate}
  \item If $\ell_1=\ell_{2}=\cdots=\ell_h$, then
$$\beta_{i,j}(I_{B_{\underline{\ell},h}})=\beta_{i,j}(\mathrm{in}_{\prec}(I_{B_{\underline{\ell},h}}))=\left\{
                                                          \begin{array}{ll}
                                                            (i+1)\binom{h}{i+2}, & \hbox{$j=\ell(i+2)$;} \\
                                                            0, & \hbox{otherwise.}
                                                          \end{array}
                                                        \right.
$$
Here, $\ell=\ell_1$.
  \item If $\ell_1\neq\ell_{2}=\cdots=\ell_h$, then $$\beta_{i,j}(I_{B_{\underline{\ell},h}})=\beta_{i,j}(\mathrm{in}_{\prec}(I_{B_{\underline{\ell},h}}))=\left\{
                                                          \begin{array}{ll}
                                                            (i+1)\binom{h-1}{i+2} & \hbox{$j=\ell(i+2)$;} \\
(h-1)\binom{h-2}{i}
& \hbox{$j=\ell(i+1)+q$;}\\
                                                            0 & \hbox{otherwise.}

                                                          \end{array}
                                                        \right.
$$
Here, $q=\ell_1$ and $\ell=\ell_{2}$.
\end{enumerate}

\begin{proof} Let $J$ denote the monomial ideal $\mathrm{in}_{\prec}(I_{B_{\underline{\ell},h}})$. In both cases, $J$ has $\ell$-regular quotients  with respect to the given order by Proposition~\ref{main3}.

(1) By Corollary~\ref{main11} together with Proposition~\ref{main3}, we have  $\beta_{i,j}(J)=0$ for $j\neq \ell (i+2)$; and
\begin{align*}
\beta_{i,\ell(i+2)}(J)&=\sum_{k=1}^{\frac{h(h-1)}{2}} \binom{r_{k}}{i} =\binom{0}{i} + 2\binom{1}{i} + 3\binom{2}{i} + \ldots + (h-1)\binom{h-2}{i}\\
&=\sum_{q=1}^{h-1} q\binom{q-1}{i} = (i+1)\sum_{q=1}^{h-1} \binom{q}{i+1}
 = (i+1)\binom{h}{i+2}.
\end{align*}
Here, the second-to-last equality follows from the identity $k\binom{n}{k} =
n\binom{n-1}{k-1}$, and the last equality follows from the identity
$\sum_{j=0}^{n}\binom{j}{k}=\binom{n+1}{k+1}$ for $k\geq 0$.

It follows from  Lemma~\ref{pure} that $\beta_{i,j}(I_{B_{\underline{\ell},h}})=\beta_{i,j}(J)$ for all $i,j\geq 0$.

(2) In view of  Proposition~\ref{main3},  the sequence of degrees of generators of $J$ in this case is as follows: $$\underbrace{2\ell, \ldots, 2\ell}_{\frac{(h-2)(h-1)}{2}}, \underbrace{\ell+q,\ldots,\ell+q}_{h-1}.$$  This implies by Theorem~\ref{bridge} that $\beta_{i,j}(I_{B_{\underline{\ell},h}})=\beta_{i,j}(J)$ for all $i,j$.

By Corollary~\ref{main11} together with Proposition~\ref{main3}, we have  $\beta_{i,j}(J)=0$ for $j\notin \{\ell (i+2),\ell(i+1)+q\}$; and
\begin{align*}
\beta_{i,\ell(i+2)}(J)&=\sum_{k=1}^{\frac{(h-1)(h-2)}{2}} \binom{r_{k}}{i} =\binom{0}{i} + 2\binom{1}{i} + 3\binom{2}{i} + \cdots + (h-2)\binom{h-3}{i}\\
&=\sum_{q=1}^{h-2} q\binom{q-1}{i} = (i+1)\sum_{q=1}^{h-2} \binom{q}{i+1}
 = (i+1)\binom{h-1}{i+2};
\end{align*}

$$\beta_{i,\ell i+\ell+q}(J)=\sum_{k=\frac{(h-1)(h-2)}{2}+1}^{\frac{h(h-1)}{2}} \binom{r_{k}}{i}=(h-1)\binom{h-2}{i}.\qquad \qquad\qquad\qquad \qquad$${\hfill $\square$\par}
\end{proof}
\end{Theorem}

\section{Graphs $B_{\ell,h}^{s,t}$ and their toric ideals}

Let $B_{\ell,h}$ denote the graph $B_{\underline{\ell},h}$ defined in the preceding section if $\underline{\ell}=(\ell,\ldots,\ell)$.
In this section, we construct a new family of graphs by  attaching two odd cycles of length $2s+1$ and $2t+1$  at the vertex $y_1$ and $y_2$ of $B_{\ell,h}$ respectively,  which we denote by $B_{\ell,h}^{s,t}$.  The main objective of our computations is to determine the Betti numbers of the edge ring $\mathbb{K}[B_{\ell,h}^{s,t}]$. If $\ell=s=t=1$, the graph $B_{\ell,h}^{s,t}$ is exactly the graph $G_h$ defined in \cite[Definition 3.1]{FKV} and thus the findings in this section can be seen as extensions of certain existing results in \cite{FKV}.
\begin{Definition}
Let $\ell,s,t\geq 1$  and $h\geq 2$ be integers. The graph $B_{\ell,h}^{s,t}$ is the  graph which has the vertex set:
$$V=V_{\ell,h}\cup \{z_1,\cdots,z_{2s},w_1,\cdots,w_{2t}\}$$
and the  edge set
\begin{align*}
  E & =E_{\ell,h}\cup \{\{y_1,z_1\},\{z_i,z_{i+1}~|~1\leq i\leq 2s-1\},\{z_{2s},y_1\}\} \\
   & \quad \cup \{\{y_2,w_1\},\{w_i,w_{i+1}~|~1\leq i\leq 2t-1\},\{w_{2t},y_2\}\}.
\end{align*}
\end{Definition}
\noindent We label the edge of $B_{\ell,h}^{s,t}$ as follows: $e_{1,j}=\{y_1,x_{1,j}\},e_{i,j}=\{x_{i-1,j},x_{i,j}\},e_{2\ell,j}=\{y_2,x_{2\ell,j}\}$, for $i\in \{2,\ldots,2\ell-1\}$ and $j\in \{1,\ldots,h\}$. $f_1=\{y_1,z_1\}, f_2=\{z_1,z_2\},\ldots$, $f_{2s}=\{z_{2s-1},z_{2s}\}, f_{2s+1}=\{z_{2s},y_1\}$,
$g_1=\{y_2,w_1\}, g_2=\{w_1,w_2\},\ldots,g_{2t}=\{w_{2t-1},w_{2t}\}$,
$g_{2t+1}=\{w_{2t},y_2\}$.
\begin{figure}[htbp]%
    \centering

        \begin{tikzpicture}[thick, scale=0.8, every node/.style={scale=0.8}]]
        \fill (0,0) circle (.07);
        \fill (10,0) circle (.07);
        \fill (2,3) circle (.07);
        \fill (4,3) circle (.07);
        \fill (6,3) circle (.07);
        \fill (8,3) circle (.07);
        \fill (2,2) circle (.07);
        \fill (4,2) circle (.07);
        \fill (6,2) circle (.07);
        \fill (8,2) circle (.07);
        \fill (2,1) circle (.07);
        \fill (4,1) circle (.07);
        \fill (6,1) circle (.07);
        \fill (8,1) circle (.07);
        \fill (2,-1) circle (.07);
        \fill (4,-1) circle (.07);
        \fill (6,-1) circle (.07);
        \fill (8,-1) circle (.07);
        \fill (8,-2) circle (.07);
        \fill (2,-2) circle (.07);
        \fill (4,-2) circle (.07);
        \fill (6,-2) circle (.07);
        \fill (8,-2) circle (.07);
        \fill (2,-3) circle (.07);
        \fill (4,-3) circle (.07);
        \fill (6,-3) circle (.07);
        \fill (8,-3) circle (.07);
        \fill (-1.5,-1.5) circle (.07);
        \fill (-1.5,1.5) circle (.07);
        \fill (-3.5,-1.5) circle (.07);
        \fill (-3.5,1.5) circle (.07);
        \fill (11.5,-1.5) circle (.07);
        \fill (11.5,1.5) circle (.07);
        \fill (13.5,-1.5) circle (.07);
        \fill (13.5,1.5) circle (.07);
        \draw[thick] (0,0) -- (2,3);
        \draw[thick] (2,3) -- (4,3);
        \draw[thick,dashed] (4,3) -- (6,3);
        \draw[thick] (6,3) -- (8,3);
        \draw[thick] (8,3) -- (10,0);
        \draw[thick] (0,0) -- (2,2);
        \draw[thick] (2,2) -- (4,2);
        \draw[thick,dashed] (4,2) -- (6,2);
        \draw[thick] (6,2) -- (8,2);
        \draw[thick] (8,2) -- (10,0);
        \draw[thick] (0,0) -- (2,1);
        \draw[thick] (2,1) -- (4,1);
        \draw[thick,dashed] (4,1) -- (6,1);
        \draw[thick] (6,1) -- (8,1);
        \draw[thick] (8,1) -- (10,0);
        \draw[thick] (0,0) -- (2,-1);
        \draw[thick] (2,-1) -- (4,-1);
        \draw[thick,dashed] (4,-1) -- (6,-1);
        \draw[thick] (6,-1) -- (8,-1);
        \draw[thick] (8,-1) -- (10,0);
        \draw[thick] (0,0) -- (2,-2);;
        \draw[thick] (2,-2) -- (4,-2);
        \draw[thick,dashed] (4,-2) -- (6,-2);
        \draw[thick] (6,-2) -- (8,-2);
        \draw[thick] (8,-2) -- (10,0);
        \draw[thick] (0,0) -- (2,-3);
        \draw[thick] (2,-3) -- (4,-3);
        \draw[thick,dashed] (4,-3) -- (6,-3);
        \draw[thick] (6,-3) -- (8,-3);
        \draw[thick] (8,-3) -- (10,0);
        \draw[very thick,dotted] (3,0.4) -- (3,-0.4);
        \draw[very thick,dotted] (7,0.4) -- (7,-0.4);
        \draw[thick] (0,0) -- (-1.5,-1.5);
        \draw[thick] (-1.5,-1.5) -- (-3.5,-1.5);
        \draw[thick] (0,0) -- (-1.5,1.5);
        \draw[thick] (-1.5,1.5) -- (-3.5,1.5);
        \draw[thick,dashed] (-3.5,-1.5) -- (-3.5,1.5);
        \draw[thick] (10,0) -- (11.5,-1.5);
        \draw[thick] (11.5,-1.5) -- (13.5,-1.5);
        \draw[thick] (10,0) -- (11.5,1.5);
        \draw[thick] (11.5,1.5) -- (13.5,1.5);
        \draw[thick,dashed] (13.5,-1.5) -- (13.5,1.5);
        \draw (1.3,1.8) node[anchor=south east]{$e_{1,1}$};
        \draw (1.5,1) node[anchor=south east]{$e_{1,2}$};
        \draw (1.5,0.5) node[anchor=south east]{$e_{1,3}$};
        \draw (1.8,-0.5) node[anchor=south east]{$e_{1,h-2}$};
        \draw (1.8,-1.3) node[anchor=south east]{$e_{1,h-1}$};
        \draw (1.3,-2) node[anchor=south east]{$e_{1,h}$};
        \draw (3,3) node[anchor=south]{$e_{2,1}$};
        \draw (3,2) node[anchor=south]{$e_{2,2}$};
        \draw (3,1) node[anchor=south]{$e_{2,3}$};
        \draw (3,-1) node[anchor=north]{$e_{2,h-2}$};
        \draw (3,-2) node[anchor=north]{$e_{2,h-1}$};
        \draw (3,-3) node[anchor=north]{$e_{2,h}$};
        \draw (7,3) node[anchor=south]{$e_{2\ell_1-1,1}$};
        \draw (7,2) node[anchor=south]{$e_{2\ell_2-1,2}$};
        \draw (7,1) node[anchor=south]{$e_{2\ell_3-1,3}$};
        \draw (7,-1) node[anchor=north]{$e_{2\ell_{h-2}-1,h-2}$};
        \draw (7,-2) node[anchor=north]{$e_{2\ell_{h-1}-1,h-1}$};
        \draw (7,-3) node[anchor=north]{$e_{2\ell_{h}-1,h}$};
        \draw (8.7,1.8) node[anchor=south west]{$e_{2\ell_1,1}$};
        \draw (8.2,1.3) node[anchor=south west]{$e_{2\ell_2,2}$};
        \draw (8.2,0.7) node[anchor=south west]{$e_{2\ell_3,3}$};
        \draw (8.2,-0.5) node[anchor=south west]{$e_{2\ell_{h-2},h-2}$};
        \draw (8.2,-1.3) node[anchor=south west]{$e_{2\ell_{h-1},h-1}$};
        \draw (8.7,-2.1) node[anchor=south west]{$e_{2\ell_h,h}$};
        \draw (0,0) node[anchor=east]{$y_1$};
        \draw (10,0) node[anchor=west]{$y_2$};
        \draw (-0.75,0.75) node[anchor=south west]{$f_1$};
        \draw (-2.5,1.5) node[anchor=south]{$f_2$};
        \draw (-0.9,-0.75) node[anchor=north west]{$f_{2s+1}$};
        \draw (-2.5,-1.5) node[anchor=north]{$f_{2s}$};
        \draw (8.7,-2.1) node[anchor=south west]{$e_{2\ell,h}$};
        \draw (10.75,0.75) node[anchor=south east]{$g_1$};
        \draw (12.5,1.5) node[anchor=south]{$g_2$};
        \draw (10.9,-0.75) node[anchor=north east]{$g_{2t+1}$};
        \draw (12.5,-1.5) node[anchor=north]{$g_{2t}$};
        \draw (5,-3.5) node[anchor=north]{Fig. 2: The graph $B_{\ell,h}^{s,t}$};
        \end{tikzpicture}

\end{figure}

Let $\prec$ denote the graded reverse lexicographic monomial order on the polynomial ring $\mathbb{K}[e_{1,1},\cdots,e_{1,h},\cdots,e_{2\ell,1},\cdots,e_{2\ell,h},f_1,\cdots,f_{2s+1},g_1,\cdots,g_{2t+1}]$,
where the variables are ordered in the following way:
\noindent

\begin{equation*}
\begin{split}
&e_{1,1}\succ\cdots\succ e_{1,h}\succ e_{3,1}\succ\cdots\succ e_{3,h}\succ\cdots\succ e_{2\ell-1,1}\cdots\succ e_{2\ell-1,h}\succ g_1\succ\cdots\succ g_{2t+1}
\\&\succ f_1\succ\cdots\succ f_{2s+1}\succ e_{2,1}\succ\cdots\succ e_{2,h}\succ e_{4,1}\succ\cdots\succ e_{4,h}\succ\cdots\succ e_{2\ell,1}\succ\cdots\succ e_{2\ell,h}.
\end{split}
\end{equation*}
For simplicity,  we set
 $$Y_{1,i}=\prod\limits_{p=1}^{l}e_{2p-1,i}, ~Y_{2,j}=\prod\limits_{p=1}^{l}e_{2p,j}, ~U=\prod\limits_{p=1}^{s}f_{2p} \prod\limits_{p=1}^{t+1}g_{2p-1} ~\mathrm{and}~ V=\prod\limits_{p=1}^{s+1}f_{2p-1}\prod\limits_{p=1}^{t}g_{2p}.$$
Here,  $1\leq i<j\leq h$. Note that $\mathrm{deg}(Y_{1,i})=\mathrm{deg}(Y_{2,j})=\ell$ and $\mathrm{deg}(U)=\mathrm{deg}(V)=s+t+1$.
\begin{Lemma}
 Let $\ell,r,s\geq 1$ and $h\geq 2$ be integers. Then $\mathcal{G}_{\ell,h}^{s,t}=\mathcal{G}_1\cup\mathcal{G}_2\cup\mathcal{G}_3$ is the universal Gr\"{o}bner basis of the toric ideal $I_{B_{\ell,h}^{s,t}}$, where
\begin{align*}
  {\mathcal{G}_1}= & ~\{Y_{1,j}Y_{2,i}-Y_{1,i}Y_{2,j}\:\; 1\leq i< j\leq h\};\\
\mathcal{G}_2= & ~\{Y_{1,i}Y_{1,j}U-Y_{2,i}Y_{2,j}V \:\; 1\leq i< j\leq h\};\\
\mathcal{G}_3= & ~ \{Y_{1,i}^2U-Y_{2,i}^2V\:\; 1\leq i\leq h\}.
\end{align*}
In particular, the minimal generating set of $\mathrm{in}_{\prec}(I_{B_{\ell,h}^{s,t}})$  is
$$\{Y_{1,j}Y_{2,i}\:\; 1\leq i<j\leq h\}\cup\{Y_{1,i}Y_{1,j}U\:\; 1\leq i<j\leq h\}\cup\{Y_{1,i}^2U\:\; 1\leq i\leq h\}.$$

\begin{proof}
 According to \cite[Lemma~3.2]{OH}, the set of primitive even closed walk  $\mathcal{PC}(B_{\ell,h}^{s,t})$ is the disjoint union $W_1\cup W_2\cup W_3$,  where
\begin{align*}
  W_1 & =\{e_{1,i},e_{2,i},\cdots,e_{2\ell,i},e_{2\ell,j},\cdots,e_{2,j},e_{1,j}\:\; 1\leq i<j\leq h\} \\
  W_2 & =\{f_1,\cdots,f_{2s+1},e_{1,i},e_{2,i},\cdots,e_{2\ell,i},g_1,\cdots,g_{2t+1},e_{2\ell,j},\cdots,e_{2,j},e_{1,j}\:\; 1\leq i<j\leq h\} \\
  W_3 & =\{f_1,\cdots,f_{2s+1},e_{1,i},e_{2,i},\cdots,e_{2\ell,i},g_1,\cdots,g_{2t+1},e_{2\ell,i},\cdots,e_{2,i},e_{1,i}\:\; 1\leq i\leq h\}
\end{align*}
This implies $\mathcal{G}_{\ell,h}^{s,t}$ is the universal Gr\"{o}bner basis of $I_{B_{\ell,h}^{s,t}}$.  In addition, it is straightforward to get the minimal generators of $\mathrm{in}_{\prec}(I_{B_{\ell,h}^{s,t}})$ from $\mathcal{G}_{\ell,h}^{s,t}$ under the above monomial order.{\hfill $\square$\par}
\end{proof}
\end{Lemma}
We label the generators of $\mathrm{in}_{\prec}(I_{B_{\ell,h}^{s,t}})$ one by one as follows:
\begin{equation*}\begin{split} & m_1=Y_{1,h}Y_{2,h-1},\quad m_2=Y_{1,h}Y_{2,h-2},\quad  m_3=Y_{1,h-1}Y_{2,h-2},  \ldots,  m_{\frac{h(h-1)}{2}}=Y_{1,2}Y_{2,1},\\
&m_{\frac{h(h-1)}{2}+1}=Y_{1,h}Y_{1,h-1}U,\quad m_{\frac{h(h-1)}{2}+2}=Y_{1,h}Y_{1,h-2}U,  \ldots, m_{h(h-1)}=Y_{1,2}Y_{1,1}U,\\ & m_{h(h-1)+1}=Y_{1,h}^2U, \quad \ldots, \quad m_{h^2}=Y_{1,1}^2U.\end{split}\end{equation*}
Then $\mathrm{in}_{\prec}(I_{B_{\ell,h}^{s,t}})=(m_1,\ldots,m_{h^2})$ and we will show that $\mathrm{in}_{\prec}(I_{B_{\ell,h}^{s,t}})$ has $\ell$-regular quotients with respect to this order of generators.
\begin{Proposition}\label{quotients2}
 Let $\ell,s,t\geq 1$ and $h\geq 2$ be integers.  The initial ideal  $\mathrm{in}_{\prec}(I_{B_{\ell,h}^{s,t}})=(m_1,\ldots,m_{h^2})$ has $\ell$-regular quotients with respect to this order of generators. More precisely, the colon ideal $(m_1,\ldots, m_{k-1}):m_k$ is generated by a  regular sequence of length $r_k$ in a single degree $\ell$ for $k=2,\ldots,h^2$. Here,
$$(r_2,\ldots,r_{\frac{h(h-1)}{2}})=(\underbrace{{1,1}}_{2},\underbrace{2,2,2}_{3},\ldots,\underbrace{h-2,h-2,\ldots,h-2}_{h-1});$$
 If $\frac{h(h-1)}{2}+1\leq k\leq h^2$, then $h-1\leq r_k\leq 2h-2$. Moreover, for any $p$ with $h-1\leq p \leq 2h-2$, one has
 $$~|\{\frac{h(h-1)}{2}+1\leq k\leq h^2 \:\;  r_k=p\}|=2h-p-1.$$
\begin{proof}
\noindent\textit{Case 1.} For $2\leq k\leq \frac{h(h-1)}{2}$, the result follows in a similar way as in the proof of Proposition~\ref{main3}.

\noindent\textit{Case 2.} If $k\in\{\frac{h(h-1)}{2}+1,\cdots,h(h-1)\}$, then  $m_k=Y_{1,j}Y_{1,i}U$ for some $1\leq i< j\leq h$. If $j<h$, then
\begin{align*}
  L_k & =(Y_{1,h}Y_{2,h-1}, \ldots, Y_{1,2}Y_{2,1},Y_{1,h}Y_{1,h-1}U,\ldots,Y_{1,j+1}Y_{1,i}U):Y_{1,j}Y_{1,i}U \\
   & =(Y_{1,i+1},\ldots,\widehat{Y_{1,j}},\ldots,Y_{1,h},Y_{2,1},\ldots,Y_{2,j-1}).
\end{align*}
Here, $\widehat{Y_{1,j}}$ represents the removal of the element $Y_{1,j}$ from the set of generators. Thus we get $r_k=h-i+j-2$.
If $j=h$, then
\begin{align*}
  L_k & =(Y_{1,h}Y_{2,h-1}, \ldots, Y_{1,2}Y_{2,1},Y_{1,h}Y_{1,h-1}U,\ldots,Y_{1,i+2}Y_{1,i+1}U):Y_{1,h}Y_{1,i}U \\
   & =(Y_{1,i+1},\ldots,\ldots,Y_{1,h-1},Y_{2,1},\ldots,Y_{2,h-1}).
\end{align*}
 Thus we also get $r_k=2h-i-2=h+j-i-2$.

\noindent\textit{Case 3.} If $k\in\{h(h-1)+1,\cdots,h^2\}$, then $m_k=Y_{1,i}^2U$ for some $1\leq i\leq h$,  and
\begin{align*}
  L_k & =(Y_{1,h}Y_{2,h-1}, \ldots, Y_{1,2}Y_{1,1}U, Y_{1,h}^2U, \ldots, Y_{1,i+1}^2U):Y_{1,i}^2U \\
   & =(Y_{1,1}, \ldots, \widehat{Y_{1,i}}, \ldots,  Y_{1,h}, Y_{2,1},\ldots,Y_{2,i-1}).
\end{align*}
Here, $\widehat{Y_{1,i}}$ represents the removal of the element $Y_{1,i}$ from the set of generators. It follows that $r_k=h+i-2$.

 By  the above discussion on the three cases, it follows that ${\rm in}_{\prec}(I_{B_{\ell,h}^{s,t}})$ has $\ell$-regular quotients with respect to the given order. Integrating the discussions on case 2 and case 3, we see that   if $\frac{h(h-1)}{2}+1\leq k\leq h^2$, then $h-1\leq r_k\leq 2h-2$. Moreover, for any $h-1\leq p\leq  2h-2$, one has  \begin{align*}|\{\frac{h(h-1)}{2}+1\leq k\leq h^2 \:\;  r_k=p\}|&=1+|\{(i,j)\:\; j-i=p-h+2, 1\leq i<j\leq h\}| \\ &=2h-p-1.\end{align*} This completes the proof.{\hfill $\square$\par}
\end{proof}
\end{Proposition}

\begin{Theorem}\label{betti2}
Fix integers $\ell,s,t\geq 1$ and $h\geq 2$. Then the graded Betti numbers of $I_{B_{\ell,h}^{s,t}}$ and ${\rm in}_{\prec}(I_{B_{\ell,h}^{s,t}})$ are equal, and the explicit graded Betti numbers are:
$$\beta_{i,j}(I_{B_{\ell,h}^{s,t}})=\left\{
                        \begin{array}{ll}
                          (i+1)\binom{h}{i+2}, & \hbox{$j=\ell(i+2)$;} \\

                              \binom{2h}{i+2}-(i+3)\binom{h}{i+2},                        & \hbox{$j=\ell(i+2)+s+t+1$;} \\
                          0, & \hbox{otherwise.}
                        \end{array}
                      \right.
$$
\begin{proof}
Proposition~\ref{quotients2} shows ${\rm in}_{\prec}(I_{B_{\ell,h}^{s,t}})$ has $\ell$-regular quotients with respect to the  ordered set $\{m_1,\ldots, m_{h^2}\}$. Note that $\mathrm{deg}(m_1)=\cdots=\mathrm{deg}(m_{\frac{h(h-1)}{2}})=2\ell$ and $\mathrm{deg}(m_{\frac{h(h-1)}{2}+1})=\cdots=\mathrm{deg}(m_{h^2})=2\ell+s+t+1$.  Since $B_{\ell,h}$ is an induced graph of $B_{\ell,h}^{s,t}$ and since $\mathrm{in}_{\prec}(I_{B_{\ell,h}})=(m_1,\ldots, m_{\frac{h(h-1)}{2}})$,  it follows that $\beta_{i,j}(I_{B_{\ell,h}^{s,t}})=\beta_{i,j}({\rm in}_{\prec}(I_{B_{\ell,h}^{s,t}}))$ for all $i,j\geq0$ by Theorem~\ref{bridge}.

 We still need to calculate the graded betti number of ${\rm in}_{\prec}(I_{B_{\ell,h}^{s,t}})$ for all $i,j\geq0$.
By Theorem~\ref{main1}, we see that $\beta_{i,j}({\rm in}_{\prec}(I_{B_{\ell,h}^{s,t}}))=0$ for  $j\notin\{ \ell(i+2),\ell(i+2)+s+t+1\}$, and that $\beta_{i,j}({\rm in}_{\prec}(I_{B_{\ell,h}^{s,t}}))=\beta_{i,j}({\rm in}_{\prec}(I_{B_{\underline{\ell},h}}))=(i+1)\binom{h}{i+2}$  if $j=\ell(i+2)$.

 If $j=\ell(i+2)+s+t+1$, it follows from Proposition~\ref{quotients2} that
\begin{equation*}
\begin{split}
\beta_{i,j}({\rm in}_{\prec}(I_{B_{\ell,h}^{s,t}}))&=\sum\limits_{p=h-1}^{2h-2}(2h-1-p)\binom{p}{i}\\
&=\sum\limits_{p=h-1}^{2h-2}2h\binom{p}{i}-\sum\limits_{p=h-1}^{2h-2}(i+1)\binom{p+1}{i+1}\\
&=2h\binom{2h-1}{i+1}-2h\binom{h-1}{i+1}-(i+1)\binom{2h}{i+2}+(i+1)\binom{h}{i+2}\\
&=\binom{2h}{i+2}-(i+3)\binom{h}{i+2}.
\end{split}
\end{equation*}
\noindent Here, we employ the equation $\sum_{j=0}^{n}\binom{j}{i}=\binom{n+1}{i+1}$  multiple times. Therefore, the proof is concluded.{\hfill $\square$\par}
\end{proof}
\end{Theorem}

\begin{Corollary}
Fix integers $\ell,s,t\geq 1$ and $h\geq 2$. Then we have
\begin{enumerate}
  \item ${\rm{pdim}}(I_{B_{\ell,h}^{s,t}})=2h-2$;
  \item ${\rm{reg}}(I_{B_{\ell,h}^{s,t}})=2\ell h+s+t-2h-1$;
  \item ${\rm{deg}}~h_{\mathbb{K}[B_{\ell,h}^{s,t}]}(t)=(2\ell-1)h+s+t+1$.
\end{enumerate}
\begin{proof}
The assertions (1) and (2) follow from either Theorem~\ref{betti2} or Proposition~\ref{quotients2} directly. In the same vein, we have $\beta_{2h-2,2\ell h+s+t+1}(I_{B_{\ell,h}^{s,t}})$ is the unique extremal  Betti number of $I_{B_{\ell,h}^{s,t}}$ and hence $\beta_{2h-1,2\ell h+s+t+1}(\mathbb{K}[B_{\ell,h}^{s,t}])$ is the unique extremal  Betti number of $\mathbb{K}[B_{\ell,h}^{s,t}]$.

 Since the graph $B_{\ell,h}^{s,t}$ has $2\ell h+2s+2t+2$ edges and $2\ell h+2s+2t-h+2$ vertices, we have $\mathbb{K}[E(B_{\ell,h}^{s,t})]$ is a polynomial ring in $2\ell h+2s+2t+2$ variables, and
${\rm{dim}}(\mathbb{K}[B_{\ell,h}^{s,t}])=2\ell h+2s+2t-h+2$ by \cite[Corollary~10.1.21]{V1}. It follows that  $${\rm{deg}}~h_{\mathbb{K}[B_{\ell,h}^{s,t}]}(t)=2\ell h+s+t+1+2\ell h+2s+2t-h+2-(2\ell h+2s+2t+2),$$
as required.{\hfill $\square$\par}
\end{proof}
\end{Corollary}

{\bf \noindent Acknowledgment:} This project is supported by NSFC (No. 11971338)
\vspace{5mm}

{\bf \noindent Data availability:} The data used to support the findings of this study are included within the article.

\vspace{5mm}

{\bf \noindent Statement:}    On behalf of all authors, the corresponding author states that there is no conflict of interest.


\begin{thebibliography}{9999}
\bibitem{BOV}
J.~Biermann, A.~O'Keefe, A.~Van Tuyl, \textit {Bounds on the regularity of toric ideals of graphs}, Advances in Applied Mathematics, {\bf85}~(2017),~84-102.  \par

\bibitem{C} Ri-xiang Chen, \textit{Minimal free resolutions of linear edge ideals}, Journal of Algebra, {\bf324}~(2010),~3591-3613.
\bibitem{EH}
V.~Ene,~J.~Herzog, \textit {Gr\"{o}bner Bases in Commutative Algebra,} Graduate Studies in
Mathematics, Vol. 130 (American Mathematical Society, 2012).
\bibitem{FKV}
G.~Favacchio, G.~Keiper, A.~Van Tuyl, \textit{Regularity and h-polynomials of toric ideals of graphs}, Proceedings of the American Mathematical Society, {\bf148}~(2020),~4665-4677.
\bibitem{FHKV}
G.~Favacchio, J.~Hofscheier, G.~Keiper, and A.~Van Tuyl, \textit{Splittings of toric ideals}, Journal of Algebra, {\bf574}~(2021),~409-433.
\bibitem{GHK}
F.~Galetto, J.~Hofscheier, G.~Keiper, C.~Kohne, A.~V.~ Tuyl, M.~E.~U. Paczka, \textit{Betti numbers of toric ideals of graphs: a case study}, Journal of Algebra and its Applications, {\bf18}~(2019),~1950226.


\bibitem{SJZ10} A. S.~Jahan,~X. Zheng, \textit{Ideals with linear quotients}, J. Combin. Theory Ser. A  {\bf117}~ (2010), 104-110.
\bibitem{HBO}
H.~T.~H\`{a}, S.~K.~Beyarslan, A.~O'Keefe, \textit{Algebraic properties of toric rings of graphs}, Communications in Algebra, {\bf47}~(2019),~1-16.
\bibitem{HHO}
J.~Herzog, T.~Hibi, H.~Ohsugi, \textit {Binomial ideals}, Cham: Springer~(2018).
\bibitem{HH}
J.~Herzog,  T.~Hibi, \textit {Monomial Ideals}, Graduate Text in Mathematics  260, Springer~(2011).
\bibitem{HT02} J. Herzog, Y. Takayama, \textit{Resolutions by mapping cones}, The Roos Festschrift volume, 2,Homology Homotopy Appl. {\bf4} (2002), 277--294.
    \bibitem{HHZ04} J.~Herzog, T.~ Hibi,  X.~Zheng, \textit{ Monomial ideals whose powers have a linear resolution}, Math. Scand. {\bf 95} (2004), 23-32.
\bibitem{NH} N.Erey, T.Hibi, \textit{The size of Betti tables of edge ideals arising from bipartite graphs}, Proceedings of the American Mathematical Society, {\bf 150}(2022), 5073--5083

\bibitem{NN}
R.~Nandi, R.~Nanduri, \textit {Betti numbers of toric algebras of certain bipartite graphs}, Journal of Algebra and Its Applications, {\bf18} (2019),~1950231.
\bibitem{OHH02} H.~Ohsugi, J.~Herzog, T.~Hibi,   \textit{Combinatorial pure rings}, Osaka J. Math. {\bf37} (2000), 745--757
\bibitem{OH}
H.~Ohsugi, T.~Hibi, \textit {Toric ideals generated by quadratic binomials}, Journal of Algebra, {\bf218} (1999),~509-527.
\bibitem{P}
I.~Peeva, \textit {Graded syzygies}, Algebra and Applications, vol. 14, Springer-Verlag London,
Ltd., London~(2011).


\bibitem{SV}
L.~Sharifan, M.~Varbaro, \textit {Graded Betti numbers of ideals with linear quotients},  Le Matematiche, V. LXIII (2008)-Fasc. II,  257--265


\bibitem{V1}
R.~H.~Villarreal, \textit {Monomial algebras}, 2nd ed., Monographs and Research Notes in Mathematics, CRC Press, Boca Raton, FL~(2015).



\end{thebibliography}
\end{document}